\documentclass[10pt]{amsart}

\usepackage{amssymb}
\usepackage{amsthm}
\usepackage{amsmath}

\usepackage[usenames]{color}
\definecolor{ANDREW}{RGB}{255,127,0}

\usepackage{tikz}
\usetikzlibrary{arrows,calc}

\usepackage[foot]{amsaddr}

\usepackage{hyperref}

\theoremstyle{plain}
\newtheorem{proposition}{Proposition}[section]
\newtheorem{theorem}[proposition]{Theorem}
\newtheorem{lemma}[proposition]{Lemma}
\newtheorem{corollary}[proposition]{Corollary}
\theoremstyle{definition}

\newtheorem{definition}[proposition]{Definition}

\newtheorem*{assumption}{Assumption}
\newtheorem*{addassumption}{Additional Assumption}
\theoremstyle{remark}
\newtheorem{remark}[proposition]{Remark}

\DeclareMathOperator{\Aff}{Aff}
\DeclareMathOperator{\Aut}{Aut}

\DeclareMathOperator{\Real}{Re}
\DeclareMathOperator{\Imaginary}{Im}

\DeclareMathOperator{\SL}{SL}
\DeclareMathOperator{\GL}{GL}

\DeclareMathOperator{\SO}{SO}
\DeclareMathOperator{\SU}{SU}

\DeclareMathOperator{\Hol}{Hol}

\DeclareMathOperator{\Span}{Span} 
 
\DeclareMathOperator{\Euc}{Euc}

\DeclareMathOperator{\Haus}{Haus}
\DeclareMathOperator{\Sp}{Sp}
\DeclareMathOperator{\Ad}{Ad}

\DeclareMathOperator{\Ac}{\mathcal{A}}

\DeclareMathOperator{\Ec}{\mathcal{E}}

\DeclareMathOperator{\Gc}{\mathcal{G}}
\DeclareMathOperator{\Hc}{\mathcal{H}}

\DeclareMathOperator{\Lc}{\mathcal{L}}

\DeclareMathOperator{\Nc}{\mathcal{N}}
\DeclareMathOperator{\Oc}{\mathcal{O}}
\DeclareMathOperator{\Pc}{\mathcal{P}}

\DeclareMathOperator{\Bb}{\mathbb{B}}
\DeclareMathOperator{\Cb}{\mathbb{C}}
\DeclareMathOperator{\Db}{\mathbb{D}}

\DeclareMathOperator{\Kb}{\mathbb{K}}

\DeclareMathOperator{\Nb}{\mathbb{N}}

\DeclareMathOperator{\Rb}{\mathbb{R}}
\DeclareMathOperator{\Xb}{\mathbb{X}}

\DeclareMathOperator{\Zb}{\mathbb{Z}}

\DeclareMathOperator{\gL}{\mathfrak{g}}

\newcommand{\abs}[1]{\left|#1\right|}

\newcommand{\norm}[1]{\left\|#1\right\|}

\newcommand{\wh}[1]{\widehat{#1}}

\begin{document}

\title[The automorphism group and limit set of a bounded domain]{The automorphism group and limit set of a bounded domain II: the convex  case}
\author{Andrew Zimmer}\address{Department of Mathematics, College of William and Mary, Williamsburg, VA, 23185.}
\email{amzimmer@wm.edu}
\date{\today}

\maketitle

\begin{abstract} For convex domains with $C^{1,\epsilon}$ boundary we give a precise description of the automorphism group: if an orbit of the automorphism group accumulates on at least two different closed complex faces of the boundary, then the automorphism group has finitely many components and the connected component of the identity is the almost direct product of a compact group and a non-compact connected simple Lie group with real rank one and finite center. In this case, we also show the limit set is homeomorphic to a sphere and prove a gap theorem:  either the domain is biholomorphic to the unit ball (and the limit set is the entire boundary) or the limit set has co-dimension at least two in the boundary.\end{abstract}

\section{Introduction} 

Given a domain $\Omega \subset \Cb^d$, let $\Aut(\Omega)$ denote the automorphism group of $\Omega$, that is the group of biholomorphic maps $\Omega \rightarrow \Omega$. The group $\Aut(\Omega)$ is a topological group when endowed with the compact-open topology and when $\Omega$ is bounded H. Cartan proved that $\Aut(\Omega)$ is a Lie group. We will let $\Aut_0(\Omega)$ denote the connected component of the identity in $\Aut(\Omega)$.  The \emph{limit set of $\Omega$}, denoted $\Lc(\Omega)$, is the set of points $x \in \partial \Omega$ where there exists some $z \in \Omega$ and a sequence $\varphi_n \in \Aut(\Omega)$ such that $\varphi_n(z) \rightarrow x$. When $\Omega$ is bounded, $\Aut(\Omega)$ acts properly on $\Omega$. Hence for bounded domains, $\Lc(\Omega)$ is non-empty if and only if $\Aut(\Omega)$ is non-compact. 

This is the second of a series of papers studying the group $\Aut(\Omega)$ and the set $\Lc(\Omega)$. As in~\cite{Z2017b} our motivating examples are the so-called \emph{generalized ellipses}: 
\begin{align*}
\Ec_{m_1, \dots, m_d} = \left\{ (z_1,\dots, z_d)  \in \Cb^{d} : \abs{z_1}^{2m_1} +\dots +\abs{z_d}^{2m_d} < 1 \right\}
\end{align*}
where $m_1, \dots, m_d \in \Nb$. Webster~\cite{W1979} showed that  $\Aut(\Ec_{m_1,\dots, m_d})$ has finitely many components and that there is a compact normal subgroup $N \leq \Aut_0(\Omega)$ such that the quotient $\Aut_0(\Omega)/N$ is biholomorphic to $\Aut(\Bb_k)$ where
\begin{align*}
k = \#\{ i : m_i = 1\}.
\end{align*}
In addition, if $e_1,\dots, e_d$ is the standard basis of $\Cb^d$ and
\begin{align*}
V = \Span_{\Cb} \{ e_i : m_i = 1\},
\end{align*}
then $\Lc(\Ec_{m_1,\dots, m_d}) = \partial \Ec_{m_1,\dots, m_d} \cap V$. In particular, $\Lc(\Ec_{m_1,\dots, m_d})$ is diffeomorphic to an odd dimensional sphere and  either $\Ec_{m_1,\dots, m_d}$ is the unit ball or 
\begin{align*}
\dim_{\Rb} \Lc(\Ec_{m_1,\dots, m_d}) \leq \dim_{\Rb} \partial \Ec_{m_1,\dots, m_d}-2.
\end{align*}

In this paper we extend these properties to convex domains with $C^{1,\epsilon}$ boundary. Before stating the main result we need two more definitions. 

Given a group $G$ and subgroups $G_1, \dots, G_n \leq G$ we say that $G$ is the \emph{almost direct product of $G_1, \dots, G_n$} if $G=G_1\cdots G_n$ and distinct pairs of $G_1,\dots, G_n$ commute and have finite intersection. 

Given a convex domain $\Omega \subset \Cb^d$ with $C^1$ boundary and $x \in \partial \Omega$, let $T_{x}^{\Cb} \partial \Omega \subset \Cb^d$ be the complex affine hyperplane tangent to $\partial\Omega$ at $x$. Then the \emph{closed complex face of $x$ in $\partial \Omega$} is the set $T_{x}^{\Cb} \partial \Omega \cap \partial \Omega$. 
 
 \begin{theorem}\label{thm:main_convex} Suppose $\Omega \subset \Cb^d$ is a bounded convex domain with $C^{1,\epsilon}$ boundary and $\Lc(\Omega)$ intersects at least two different closed complex faces of $\partial \Omega$. Then:
 \begin{enumerate}
 \item $\Aut(\Omega)$ has finitely many connected components. 
 \item $\Aut_0(\Omega)$ is the almost direct product of closed subgroups $G$ and $N$ where 
 \begin{enumerate}
 \item $N$ is compact,
 \item  $G$ is a connected simple Lie group with finite center and real rank one.
 \end{enumerate}
   \item $\Lc(\Omega)$ is homeomorphic to a positive dimensional sphere. Moreover, either 
 \begin{enumerate}
 \item $\dim \Lc(\Omega) \leq \dim \partial \Omega - 2$ or
 \item $\Lc(\Omega) = \partial \Omega$ and $\Omega$ is biholomorphic to the unit ball. 
 \end{enumerate}
 \end{enumerate}
 \end{theorem}
 
 \begin{remark} \ \begin{enumerate} 
 \item Two Lie groups are said to be \emph{locally isomorphic} if they have isomorphic Lie algebras. It follows from the classification of simple Lie groups that every simple Lie group with real rank one is locally isomorphic to one of $\SO(k,1)$, $\SU(k,1)$, $\Sp(k,1)$, or $F^{-20}_{4}$. Further these groups coincide, up to a finite quotient, with ${ \rm Isom}_0(X)$ where $X$ is real hyperbolic space, complex hyperbolic space, quaternionic hyperbolic space, or the Cayley hyperbolic plane.
  \item A theorem of Griffiths~\cite{G1971} implies that there exists examples of domains $\Omega \subset \Cb^2$ where $\Aut(\Omega)$ is infinite, discrete, and the quotient $\Aut(\Omega) \backslash \Omega$ is compact (see~\cite{GR2015} for details). The last condition implies that $\Lc(\Omega) = \partial \Omega$. These examples are never convex by a theorem of Frankel~\cite{F1989}.
 \end{enumerate}
 \end{remark}
 
The automorphism group of the unit ball $\Bb_d$ in $\Cb^d$ is locally isomorphic to $\SU(d,1)$ and it is unclear whether there exists examples of convex domains $\Omega$ where the group $G$ in the statement of Theorem~\ref{thm:main_convex} is locally isomorphic to $\SO(k,1)$, $\Sp(k,1)$, or $F^{-20}_{4}$.  For smooth convex domains this is impossible. 

\begin{theorem}\label{thm:main_convex_Cinfty}
Suppose $\Omega \subset \Cb^d$ is a bounded convex domain with $C^{\infty}$ boundary and $\Lc(\Omega)$ intersects at least two different closed complex faces of $\partial \Omega$. If $G$ is the group in the statement of Theorem~\ref{thm:main_convex}, then $G$ is locally isomorphic to $\SU(1,k)$ for some $k \geq 1$.
\end{theorem}

\begin{remark} In~\cite[Theorems 8.1 and 9.1, Proposition 10.1]{Z2017}, we proved that if $\Omega \subset \Cb^d$ is a bounded convex domain with $C^{\infty}$ boundary and $\Lc(\Omega)$ intersects at least two different closed complex faces of $\partial \Omega$, then $\Omega$ has finite type. In particular, Theorem~\ref{thm:main_convex_Cinfty} follows from~\cite[Theorem 1.2]{Z2017b}.
\end{remark}

 \subsection{Prior Work and Motivation} 
 
 As mentioned above, this is the second paper in a series of papers studying the biholomorphism group and limit set of a bounded domain. In the first paper we considered finite type domains and proved the following.
 
  \begin{theorem}\cite[Theorem 1.2]{Z2017b} Suppose $\Omega \subset \Cb^d$ is a bounded pseudoconvex domain with finite type and $\Lc(\Omega)$ contains at least two distinct points. Then:
 \begin{enumerate}
 \item $\Omega$ is biholomorphic to a weighted homogeneous polynomial domain.  
 \item $\Aut(\Omega)$ has finitely many connected components. 
 \item $\Aut(\Omega)$ is the almost direct product of closed subgroups $G$ and $N$ where 
 \begin{enumerate}
 \item $N$ is compact,
 \item $G$ is a connected Lie group with finite center and there exists an isomorphism $\rho:G/Z(G)\rightarrow \Aut(\Bb_k)$ for some $k \geq 1$.
 \end{enumerate}
   \item $\Lc(\Omega)$ is a smooth submanifold of $\partial \Omega$ and there exists an $\rho$-equivariant diffeomorphism $F:\Lc(\Omega) \rightarrow \partial \Bb_k$. In particular, $\Lc(\Omega)$ is an odd dimensional sphere and so either 
 \begin{enumerate}
 \item $\dim \Lc(\Omega) \leq \dim \partial \Omega - 2$ or
 \item $\Lc(\Omega) = \partial \Omega$ and $\Omega$ is biholomorphic to the unit ball. 
 \end{enumerate}
 \end{enumerate}
 \end{theorem}
 
 \begin{remark} A domain $\Pc$ is a \emph{weighted homogeneous polynomial domain} if 
\begin{align*}
\Pc = \{ (w,z) \in \Cb \times \Cb^{d-1} : { \rm Im}(w) > p(z) \}
\end{align*}
where $p$ is a weighted homogeneous polynomial. 
 \end{remark}
 
Theorem~\ref{thm:main_convex} can be seen as a low regularity analogue of the above theorem.  We suspect that there exist examples of bounded convex domains $\Omega$ with $C^{1,\epsilon}$ boundary where $\Lc(\Omega)$ intersects at least two closed complex faces of $\Omega$, but $\Omega$ is not biholomorphic to a domain defined by a polynomial. 

Theorem~\ref{thm:main_convex} is also motivated by a number of prior results in the literature (for example \cite{W1977, R1979, GK1987, K1992, BP1994, W1995, Z1995, IK2001, V2009, Z2017}). See Section 1.1 in~\cite{Z2017b} for a detailed discussion. 

\subsection{Structure of the paper} Sections~\ref{sec:kob_convex} through \ref{sec:proof_convex} are devoted to the proof of Theorem~\ref{thm:main_convex}. At the end of the paper, there is a brief appendix describing some basic properties of semisimple Lie groups and the symmetric spaces they act on.  

 \subsection{Outline of the Proof of Theorem~\ref{thm:main_convex}}
 
 The proof of Theorem~\ref{thm:main_convex} has three main parts. 
 
 In the first part we show that the action of $\Aut(\Omega)$ on $\Omega$  is similar to the action of a Gromov hyperbolic group on its Cayley graph. This build upon work in~\cite{Z2017} and occupies Sections~\ref{sec:kob_convex},~\ref{sec:elem_auto_convex}, and~\ref{sec:hyp_elem_convex} of this paper. However, this similarity only goes so far, for instance we are unable to show that the action of $\Aut(\Omega)$ on $\Omega$ extends to a continuous action on $\partial \Omega$. This lack of extension creates a great deal of technical complications through out the entire paper.

In the second part of the proof, we refine Frankel's rescaling method to construct certain one-parameter groups of automorphisms with nice properties. In the late 1980's Frankel developed a method for showing, under certain conditions, that $\Aut(\Omega)$ contains one-parameter subgroups. His method is very useful, but has one problem - it provides little information about the one-parameter subgroups that are produced. In Section~\ref{sec:frankel_refined}, we refine Frankel's method using our knowledge of the geometry of the Kobayashi metric to produce one-parameter subgroups with nice dynamical properties. The main purpose of this refinement is to produce many ``hyperbolic'' automorphisms which is accomplished in Section~\ref{sec:construct_more_hyp}.
 
The third part of the proof takes place in Section~\ref{sec:proof_convex}. There we combine the structure theory of Lie groups with the facts established in parts one and two. In particular, we use the geometry of the Kobayashi metric to restrict the possible solvable subgroups of $\Aut(\Omega)$. This is used to show that the solvable radical of $\Aut_0(\Omega)$ is a torus in the center of $\Aut_0(\Omega)$. Which in turn implies that $\Aut_0(\Omega)$ is the almost direct product of a compact subgroup $N$ and a semisimple Lie group $G$ with only non-compact factors. By studying Abelian subgroups we show that $G$ has real rank one and finite center. Using the fact that ${ \rm Out}(G)$ is finite, we will show that $\Aut(\Omega)$ has finitely many components. 

To establish that $\Lc(\Omega)$ is homeomorphic to a sphere we consider the symmetric space associated to $G$. Since $G$ is a simple Lie group with real rank one and finite center, it acts transitively and by isometries on a negatively curved Riemannian symmetric space $X$. We will show that any orbit of $G$ in $\Omega$ endowed with the Kobayashi metric is quasi-isometric to $X$. Further, that $\Lc(\Omega)$ is homeomorphic to the geodesic boundary of $X$, which is a sphere. Finally, to prove that $\Lc(\Omega)$ cannot have real dimension $2d-2$, we show that $G$ cannot be locally isomorphic to $\SO(1,2d-1)$. 

 \subsection{Some notations} 
 
If $(X,d)$ is a metric space, $x \in X$, and $A \subset X$, then we define
\begin{align*}
d(x,A) = \inf\{ d(x,a) : a \in A\}.
\end{align*}
Then given subsets $A,B \subset X$ we define \emph{the Hausdorff pseudo-distance between $A$ and $B$} to be
 \begin{align*}
 d^{\Haus}(A,B) := \max\left\{ \sup_{a \in A} d(a,B), \ \sup_{b \in B} d(b,A) \right\}.
 \end{align*}

Given a domain $\Omega \subset \Cb^d$ with $C^1$ boundary and $x \in \partial \Omega$ let ${ \bf n}_\Omega(x) \in \Cb^d$ be the inward pointing unit normal vector of $\partial \Omega$ at $x$.

 \subsection*{Acknowledgements} This material is based upon work supported by the National Science Foundation under grants DMS-1400919 and DMS-1760233.
 
\section{The Kobayashi metric on convex domains}\label{sec:kob_convex}

 In this expository section we recall the definition of the Kobayashi metric and state some of its properties.  For a more thorough introduction see for instance~\cite{A1989} or~\cite{K2005}.
 
 Given a domain $\Omega \subset \Cb^d$ the \emph{(infinitesimal) Kobayashi metric} is the pseudo-Finsler metric
\begin{align*}
k_{\Omega}(x;v) = \inf \left\{ \abs{\xi} : f \in \Hol(\Db, \Omega), \ f(0) = x, \ d(f)_0(\xi) = v \right\}.
\end{align*}
By a result of Royden~\cite[Proposition 3]{R1971} the Kobayashi metric is an upper semicontinuous function on $\Omega \times \Cb^d$. In particular, if $\sigma:[a,b] \rightarrow \Omega$ is an absolutely continuous curve (as a map $[a,b] \rightarrow \Cb^d$), then the function 
\begin{align*}
t \in [a,b] \rightarrow k_\Omega(\sigma(t); \sigma^\prime(t))
\end{align*}
is integrable and we can define the \emph{length of $\sigma$} to  be
\begin{align*}
\ell_\Omega(\sigma)= \int_a^b k_\Omega(\sigma(t); \sigma^\prime(t)) dt.
\end{align*}
One can then define the \emph{Kobayashi pseudo-distance} to be
\begin{multline*}
 K_\Omega(x,y) = \inf \left\{\ell_\Omega(\sigma) : \sigma\colon[a,b]
 \rightarrow \Omega \text{ is abs. cont., } \sigma(a)=x, \text{ and } \sigma(b)=y\right\}.
\end{multline*}
This definition is equivalent to the standard definition using analytic chains by a result of Venturini~\cite[Theorem 3.1]{V1989}.

When $\Omega$ is a bounded domain, $K_\Omega$ is a non-degenerate distance. Further, directly from the definition one obtains the following proposition. 

\begin{proposition} If $f:\Omega_1 \rightarrow \Omega_2$ is holomorphic, then 
\begin{align*}
k_{\Omega_2}(f(z); df_z(v)) \leq k_{\Omega_1}(z;v)
\end{align*}
for all $z \in \Omega_1$ and $v \in \Cb^d$. Moreover, 
\begin{align*}
K_{\Omega_2}(f(z_1),f(z_2)) \leq K_{\Omega_1}(z_1,z_2)
\end{align*}
for all $z_1,z_2 \in \Omega_1$. 
\end{proposition}

We will frequently use the following basic estimate. 

\begin{proposition}\label{prop:zero_dist_est} Suppose that $\Omega$ is a bounded domain and $z_n \in \Omega$ is a sequence such that $z_n \rightarrow x \in \partial \Omega$. If $w_n \in \Omega$ is sequence and
\begin{align*}
\lim_{n \rightarrow \infty} K_\Omega(z_n,w_n) =0,
\end{align*}
then $w_n \rightarrow \xi$. 
\end{proposition}

\begin{proof} Fix a bounded domain $\Omega_1$ such that $\overline{\Omega} \subset \Omega_1$. Then 
\begin{align*}
\limsup_{n \rightarrow \infty} K_{\Omega_1}(z_n,w_n) \leq \limsup_{n \rightarrow \infty} K_\Omega(z_n,w_n) =0.
\end{align*}
Since $K_{\Omega_1}$ is a metric on $\Omega_1$ we see that $w_n \rightarrow x$. 
\end{proof}

\subsection{Convex domains}

For general domains there is no known characterization of when the Kobayashi metric is Cauchy complete, but for convex domains we have the following result of Barth.

\begin{theorem}\cite{B1980}\label{thm:barth}
Suppose $\Omega$ is a convex domain. The the following are equivalent:
\begin{enumerate}
\item $\Omega$ does not contain any complex affine lines, 
\item $(\Omega, K_\Omega)$ is a Cauchy complete, proper metric space. 
\end{enumerate}
\end{theorem}

We will also frequently use the following result about the asymptotic geometry of $(\Omega, K_\Omega)$. 

\begin{proposition}\cite[Proposition 3.5]{Z2014b}\label{prop:finite_dist}
Suppose $\Omega$ is a bounded convex domain and $x, y \in \partial \Omega$ are distinct. Assume $z_m, w_n \in \Omega$ are sequences such that $z_m \rightarrow x$ and $w_n \rightarrow y$. If 
\begin{align*}
\liminf_{m,n \rightarrow \infty} K_{\Omega}(z_m, w_n) < \infty
\end{align*}
and $L$ is the complex line containing $x$ and $y$, then $L \cap \Omega = \emptyset$ and the interior of $\partial\Omega \cap L$ in $L$ contains $x$ and $y$. In particular, if $\partial \Omega$ is $C^1$, then 
\begin{align*}
T_{x}^{\Cb} \partial \Omega = T_{y}^{\Cb} \partial \Omega.
\end{align*}
\end{proposition}
 
 \subsection{The Gromov product associated to the Kobayashi metric} In a metric space $(X,d)$, the Gromov product of $x,y \in X$ at $z \in X$ is defined to be
 \begin{align*}
 (x|y)_z = \frac{1}{2} \left( d(x,z) +d(z,y) - d(x,y) \right).
 \end{align*}
 When $(X,d)$ is a proper geodesic Gromov hyperbolic metric space, there is a compactification $X \cup X(\infty)$ of $X$, called the \emph{ideal boundary}, with the following property.
 
 \begin{proposition} Suppose $(X,d)$ is a proper geodesic Gromov hyperbolic metric space. Suppose $x_m, y_n$ are sequences in $X$ such that $x_m \rightarrow \xi \in X(\infty)$ and $y_n \rightarrow \eta \in X(\infty)$. Then $\xi = \eta$ if and only if 
 \begin{align*}
 \lim_{m,n \rightarrow \infty} (x_m |y_n)_{z} = \infty
 \end{align*}
 for any $z \in X$. 
 \end{proposition}
 
 For the Kobayashi metric on convex domains the Gromov product behaves almost as nicely near the topological boundary. For a domain $\Omega \subset \Cb^d$ we define the \emph{Gromov product of $x,y$ at $z$} to be 
 \begin{align*}
 (x|y)_z^{\Omega} = \frac{1}{2} \left( K_\Omega(x,z) + K_\Omega(z,y) - K_\Omega(x,y) \right).
 \end{align*}
 Then we have the following. 
  
  \begin{theorem}\label{thm:GP}\cite[Theorem 4.1]{Z2017} Suppose $\Omega \subset \Cb^d$ is a bounded convex domain with $C^{1,\epsilon}$ boundary. Suppose $z_m, w_n$ are sequences in $\Omega$ such that $z_m \rightarrow x \in \partial \Omega$ and $w_n \rightarrow y \in \partial \Omega$. Then:
 \begin{enumerate}
 \item If $x = y$, then 
 \begin{align*}
 \lim_{m,n \rightarrow \infty} (z_m | w_n)_{z_0}^{\Omega} = \infty.
 \end{align*}
 \item If 
 \begin{align*}
 \limsup_{m,n \rightarrow \infty} (z_m | w_n)_{z_0}^{\Omega} = \infty,
 \end{align*}
  then $T_{x}^{\Cb} \partial \Omega = T_{y}^{\Cb} \partial \Omega$. 
\end{enumerate}
 \end{theorem}
 
 \subsection{Almost-geodesics}

A \emph{geodesic} in a metric space $(X,d)$ is a curve $\sigma: I \rightarrow X$ such that 
\begin{align*}
d(\sigma(s), \sigma(t))=\abs{t-s}
\end{align*}
for all $s,t \in I$. When the Kobayashi metric is Cauchy complete, every two points are joined by a geodesic. However, it is often more convenient to work with larger classes of curves.

\begin{definition}\label{def:almost_geodesic}
Suppose $\Omega \subset \Cb^d$ is a bounded domain and $I \subset \Rb$ is an interval. For $\lambda \geq 1$ and $\kappa \geq 0$ a curve $\sigma:I \rightarrow \Omega$ is called an \emph{$(\lambda, \kappa)$-almost-geodesic} if 
\begin{enumerate} 
\item for all $s,t \in I$  
\begin{align*}
\frac{1}{\lambda} \abs{t-s} - \kappa \leq K_\Omega(\sigma(s), \sigma(t)) \leq \lambda \abs{t-s} +  \kappa;
\end{align*}
\item $\sigma$ is absolutely continuous (hence $\sigma^\prime(t)$ exists for almost every $t\in I$), and for almost every $t \in I$
\begin{align*}
k_\Omega(\sigma(t); \sigma^\prime(t)) \leq \lambda e^{\kappa}.
\end{align*}
\end{enumerate}
\end{definition}

\begin{remark} In~\cite[Proposition 4.6]{BZ2017}, we proved that every geodesic in the Kobayashi metric is an $(1,0)$-almost-geodesic. \end{remark}

There are several reasons to study almost-geodesics instead of geodesics. First almost-geodesics always exist: for domains $\Omega$ where the metric space $(\Omega, K_\Omega)$ is not Cauchy complete there may not be a geodesic joining every two points, but there is always an $(1,\kappa)$-almost-geodesic joining any two points in $\Omega$, see \cite[Proposition 4.4]{BZ2017}. Further, it is sometimes possible to find explicit almost-geodesics: for convex domains with $C^{1,\epsilon}$ boundary, inward pointing normal lines can be parametrized to be almost-geodesics, see Proposition~\ref{prop:normal_lines} below. Finally, almost-geodesics are close enough to geodesics that one can establish properties about their behavior, see Theorem~\ref{thm:visible} below. 
 
\begin{proposition}\label{prop:normal_lines}\cite[Proposition 4.3]{Z2017} Suppose that $\Omega$ is a bounded convex domain with $C^{1,\epsilon}$ boundary. Assume $r > 0$ is such that 
\begin{align*}
x+ r \cdot {\bf n}_{\Omega}(x) \in \Omega
\end{align*}
for all $x \in \partial \Omega$. Then there exists $\kappa >0$ such that: for any $x \in \partial \Omega$ the curve $\sigma_x:[0,\infty) \rightarrow \Omega$ given by 
\begin{align*}
\sigma_x(t) = x+r e^{-2t} {\bf n}_{\Omega}(x)
\end{align*}
is an $(1,\kappa)$-almost-geodesic. 
\end{proposition}

For convex domains with $C^{1,\epsilon}$ boundary, we can use Theorem~\ref{thm:GP} to understand the behavior of almost geodesics. 
 
 \begin{theorem}\cite[Theorem 6.1]{Z2017}\label{thm:visible} Suppose $\kappa \geq 0$, $\Omega \subset \Cb^d$ is a bounded convex domain with $C^{1,\epsilon}$ boundary, and $x,y \in \partial \Omega$ satisfy $T_{x}^{\Cb} \partial \Omega \neq T_{y}^{\Cb} \partial \Omega$. If $U,V \subset \overline{\Omega}$ are open sets containing $T_{x}^{\Cb} \partial \Omega \cap \partial \Omega, T_{y}^{\Cb} \partial \Omega \cap \partial \Omega$ respectively and $\overline{U} \cap \overline{V} = \emptyset$, then there exists a compact set $K \subset \Omega$ such that: if $\sigma: [a,b] \rightarrow \Omega$ is an $(1,\kappa)$-almost-geodesic with $\sigma(a) \in U$ and $\sigma(b) \in V$, then 
 \begin{align*}
 \sigma([a,b]) \cap K \neq \emptyset.
 \end{align*}
 \end{theorem}
 
 \begin{remark} Informally this theorem says that almost-geodesics bend into the domain just like geodesics do in the Poincar{\'e} model of the real hyperbolic plane.
 \end{remark}
 
 \begin{proof}[Proof Sketch] Suppose not, then we can find a sequence of $(1,\kappa)$-almost-geodesics $\sigma_n : [a_n, b_n] \rightarrow \Omega$ with 
 \begin{align*}
 \lim_{n \rightarrow \infty} d_{\Euc}\left(\sigma(a_n), T_{x}^{\Cb} \partial \Omega\right) = 0 =  \lim_{n \rightarrow \infty} d_{\Euc}\left(\sigma(b_n), T_{y}^{\Cb} \partial \Omega\right)
 \end{align*}
 and $\sigma_n([a_n, b_n])$ leaves every compact set of $\Omega$. Then there exists a sequence $t_n \in [a_n, b_n]$ such that $\sigma_n(t_n) \rightarrow \tau$ and 
 \begin{align*}
 T_{\tau}^{\Cb} \partial \Omega \notin \{ T_{x}^{\Cb} \partial \Omega,  T_{y}^{\Cb} \partial \Omega\}.
 \end{align*}
 Then by Theorem~\ref{thm:GP} there exists some $M \geq 0$ such that 
 \begin{align*}
 K_\Omega(\sigma_n(a_n), \sigma_n(t_n)) \geq K_\Omega(\sigma_n(a_n), z_0)+ K_\Omega(z_0,\sigma_n(t_n)) -M
 \end{align*}
 and
  \begin{align*}
 K_\Omega(\sigma_n(b_n), \sigma_n(t_n)) \geq K_\Omega(\sigma_n(b_n), z_0)+ K_\Omega(z_0,\sigma_n(t_n)) -M.
 \end{align*}
But each $\sigma_n$ is an $(1,\kappa)$-almost-geodesic and so
  \begin{align*}
 K_\Omega(\sigma_n(a_n), \sigma_n(t_n)) & +  K_\Omega(\sigma_n(t_n), \sigma_n(b_n)) \leq K_\Omega(\sigma_n(a_n), \sigma_n(b_n)) + 3\kappa \\
 & \leq K_\Omega(\sigma_n(a_n), z_0)+K_\Omega(0, \sigma_n(b_n))+3\kappa.
 \end{align*}
So we have
 \begin{align*}
 2 K_\Omega(z_0,\sigma_n(t_n)) \leq  3 \kappa +2M,
 \end{align*}
but since $K_\Omega(\sigma_n(t_n),z_0) \rightarrow \infty$ this is impossible. 
 \end{proof}

\section{Elements of the automorphism group}\label{sec:elem_auto_convex}

For convex domains with $C^{1,\epsilon}$ boundary,  one can use Theorem~\ref{thm:GP} to establish the following analogue of the Wolff-Denjoy theorem. 

\begin{theorem}\label{thm:wolf_i}\cite[Theorem 5.1]{Z2017} Suppose $\Omega \subset \Cb^d$ is a bounded convex domain with $C^{1,\epsilon}$ boundary. If $f: \Omega \rightarrow \Omega$ is a holomorphic map, then either 
\begin{enumerate}
\item $f$ has a fixed point in $\Omega$ or 
\item there exists a point $x \in \partial \Omega$ such that 
\begin{align*}
\lim_{n \rightarrow \infty} d_{\Euc}\left(f^n(z), T_{x}^{\Cb} \partial \Omega\right) = 0
\end{align*}
for all $z \in \Omega$. 
\end{enumerate}
\end{theorem}

\begin{remark} Abate and Raissy~\cite{AR2014} proved Theorem~\ref{thm:wolf_i} with the additional assumption that $\partial\Omega$ is $C^2$. \end{remark}

Using Theorem~\ref{thm:wolf_i} we can characterize the automorphisms of $\Omega$ by the behavior of their iterates. Suppose $\Omega \subset \Cb^d$ is a bounded convex domain with $C^{1,\epsilon}$ boundary and $\varphi \in \Aut(\Omega)$. Then by Theorem~\ref{thm:wolf_i} either $\varphi$ has a fixed point in $\Omega$ or there exists a complex supporting hyperplane $H_{\varphi}^+$ of $\Omega$ such that 
\begin{align*}
 \lim_{k \rightarrow \infty} d_{Euc}\left(\varphi^k (z) , H_\varphi^+\right) = 0
 \end{align*}
 for all $z \in \Omega$. In this latter case, we call $H_{\varphi}^+$ the \emph{attracting hyperplane of} $\varphi$. 
 
  \begin{definition}\label{defn:elems_convex}
Suppose $\Omega \subset \Cb^d$ is a bounded convex domain with $C^{1,\epsilon}$ boundary and $\varphi \in \Aut(\Omega)$. Then:
\begin{enumerate}
\item $\varphi$ is \emph{elliptic} if $\varphi$ has a fixed point in $\Omega$, 
\item $\varphi$ is \emph{parabolic} if $\varphi$ has no fixed point in $\Omega$ and $H_{\varphi}^+ = H_{\varphi^{-1}}^+$,
\item $\varphi$ is \emph{hyperbolic} if $\varphi$ has no fixed points in $\Omega$ and $H_{\varphi}^+ \neq H_{\varphi^{-1}}^+$. In this case we call $H_{\varphi}^- : = H_{\varphi^{-1}}^+$  the \emph{repelling hyperplane of} $\varphi$.
\end{enumerate}
\end{definition}

\begin{remark} Theorem~\ref{thm:wolf_i} implies that every automorphism of $\Omega$ is either elliptic, hyperbolic, or parabolic. 
%\item If $G$ is a semisimple Lie group, then there are well known definitions for an element $g \in G$ to be elliptic, hyperbolic, or parabolic (see for instance~\cite[Section 1.9, Section 2.2]{E1996}). It seems possible that there exists cases where $\Aut_0(\Omega)$ is a semisimple Lie group and there exists an element $\varphi \in \Aut_0(\Omega)$ which is hyperbolic in the Lie group sense, but not $\Fc$-hyperbolic. The possible existence of such examples motivates the use of ``$\Fc$.''\end{enumerate}
\end{remark}

The rest of this section is devoted to recalling some results about the behavior of elliptic, parabolic, and hyperbolic established in~\cite{Z2017}. 
%In that paper, $\Fc$-parabolic (respectively, $\Fc$-hyperbolic) elements were just called parabolic (respectively, hyperbolic) elements. 

\subsection{Hyperbolic elements}

In a complete negatively curved Riemannian manifold a hyperbolic isometry always translates a geodesic. The next two results show that an orbit of a hyperbolic automorphism of a convex domain shadows an almost-geodesic.

\begin{theorem}\cite[Theorem 8.1]{Z2017}\label{thm:normal_line_shadowing} Suppose $\Omega$ is a bounded convex domain with $C^{1,\epsilon}$ boundary. If $h \in \Aut(\Omega)$ is a hyperbolic element, then there exists a point $x^+_{h} \in H_{h}^{+}$ such that 
\begin{align*}
K_\Omega^{\Haus}\left( \left\{ h^n z_0 : n \in\Nb \right\}, x^+_{h} + (0,r] \cdot {\bf n}_{\Omega}\left( x^+_{h}\right)  \right) < \infty
\end{align*}
for any $z_0 \in \Omega$ and $r > 0$ such that $x_h^+ + r \cdot {\bf n}_\Omega(x_h^+) \in \Omega$.
\end{theorem}

\begin{remark} For a subset $A \subset \Omega$, define
\begin{align*}
\Nc_\Omega(A;R):= \{ z \in \Omega : K_\Omega(z, A) \leq R\}.
\end{align*}
Then $K_\Omega^{\Haus}(A,B) < \infty$ if and only if there exist some $R \geq 0$ with $A \subset \Nc_\Omega(B;R)$ and $B \subset \Nc_\Omega(A;R)$. The statement of Theorem 8.1 in~\cite{Z2017} only says that 
\begin{align*}
x^+_{h}+(0,r] \cdot {\bf n}_\Omega\left( x^+_{h}\right) \subset \Nc_\Omega\left(\left\{ h^n z_0 : n \in \Nb \right\}; R_1\right)
\end{align*}
for some $r, R_1 > 0$. However, in the proof of Theorem 8.1 it is explicitly established that
\begin{align*}
\left\{ h^n z_0  : n \in \Nb \right\} \subset \Nc_\Omega\left(x^+_{h}+(0,r] \cdot {\bf n}_\Omega\left( x^+_{h}\right); R_2\right)
\end{align*}
for some $R_2 > 0$. 
\end{remark}

\begin{corollary}\label{cor:alm_geod_shadow} Suppose $\Omega$ is a bounded convex domain with $C^{1,\epsilon}$ boundary. If $h \in \Aut(\Omega)$ is a hyperbolic element, then there exists points $x^\pm_{h} \in H_{h}^{\pm}$ and an almost-geodesic $\sigma: \Rb \rightarrow \Omega$ such that 
\begin{align*}
\lim_{t \rightarrow \pm \infty} \sigma(t) = x^{\pm}_{h}
\end{align*}
and 
\begin{align*}
K_\Omega^{\Haus}\left( \left\{ h^n z_0 :n \in \Zb\right\}, \sigma(\Rb)  \right) < \infty.
\end{align*}
\end{corollary}

\begin{proof}
Let $r>0$ and $x^{\pm}_{h} \in \partial \Omega$ be as in Theorem~\ref{thm:normal_line_shadowing}. Then define the curves $\sigma^{\pm}: [0,\infty) \rightarrow \Omega$ by 
\begin{align*}
\sigma^{\pm}(t) = x^{\pm}_{h}  + r e^{-2t} {\bf n}_\Omega\left( x^{\pm}_{h} \right).
\end{align*}
By Proposition~\ref{prop:normal_lines} there exists $\kappa_0>0$ such that each $\sigma^{\pm}$ is an $(1,\kappa_0)$-almost-geodesic. Next let $\sigma_0 : [-T,T] \rightarrow \Omega$ be a geodesic with $\sigma_0(\pm T) = \sigma^{\pm}(0)$. Then define the curve
\begin{align*}
\sigma(t) = \left\{ \begin{array}{ll} 
\sigma^-(T-t) & \text{ if } t \leq -T \\
\sigma_0(t) & \text{ if } -T \leq t \leq T \\
\sigma^+(t-T) & \text{ if } T \leq t.
\end{array} \right.
\end{align*}
By~\cite[Proposition 4.6]{BZ2017}, every geodesic in the Kobayashi metric is an $(1,0)$-almost-geodesic. So $\sigma$ is absolutely continuous (as a curve $\Rb \rightarrow \Cb^d$) and 
\begin{align*}
k_\Omega(\sigma(t); \sigma^\prime(t)) \leq e^{\kappa_0}
\end{align*}
for almost every $t \in \Rb$. Further, it is easy to check that
\begin{align*}
K_\Omega(\sigma(t), \sigma(s)) \leq \abs{t-s}+2T+2\kappa_0
\end{align*}
for all $s,t \in \Rb$. Using Theorem~\ref{thm:GP} there exists some $M>0$ 
\begin{align*}
K_\Omega(\sigma(s), \sigma(t)) \geq  \abs{t - s} -M
\end{align*}
for all $s,t \in \Rb$. So $\sigma$ is an $(1,\kappa)$-almost-geodesic for some $\kappa>0$. Finally, by Theorem~\ref{thm:normal_line_shadowing} we have
\begin{equation*}
K_\Omega^{\Haus}\left( \left\{ h^n z_0 : n \in \Zb \right\}, \sigma(\Rb)  \right) < \infty. \qedhere
\end{equation*}
\end{proof}

\subsection{An uniform convergence result}

The following uniform convergence result will be helpful in many arguments that follow.

\begin{proposition}\cite[Lemma 7.5, Lemma 7.7, Proposition 7.8]{Z2017}\label{prop:non_hyp_attracting} Suppose $\Omega$ is a bounded convex domain with $C^{1,\epsilon}$ boundary. Assume that $\varphi_n \in \Aut(\Omega)$ is a sequence of non-hyperbolic elements such that $\varphi_n(z_0) \rightarrow x$ for some $z_0 \in \Omega$ and $x \in \partial \Omega$. If $U$ is a neighborhood of $T_{x}^{\Cb} \partial \Omega \cap \partial \Omega$ in $\overline{\Omega}$, then there exists some $N \geq 0$ such that 
\begin{align*}
\varphi_n(\Omega \setminus U) \subset U\text{ and } \varphi_n^{-1}(\Omega \setminus U) \subset U
\end{align*}
for all $n \geq N$.
\end{proposition}

\subsection{Continuity of attracting hyperplanes}

The next result establishes a type of continuity for the hyperplanes $H^+_{\varphi}$. 

\begin{proposition}\cite[Lemma 7.4]{Z2017}\label{prop:cont_of_att_hyp} Suppose $\Omega$ is a bounded convex domain with $C^{1,\epsilon}$ boundary. Assume that $\varphi_n \in \Aut(\Omega)$ is a sequence of non-elliptic elements such that $\varphi_n(z_0) \rightarrow x$ for some $z_0 \in \Omega$ and $x \in \partial \Omega$. Then
\begin{align*}
H_{\varphi_n}^+ \rightarrow T_{x}^{\Cb} \partial \Omega.
\end{align*}
\end{proposition}

\subsection{Constructing hyperbolic elements}

Given a subgroup $H \leq \Aut(\Omega)$, let $\Lc(\Omega; H) \subset \partial \Omega$ denote the set of points $x\in \partial\Omega$  where there exists a point $z \in \Omega$ and a sequence $\varphi_n \in H$ such that $\varphi_n(z) \rightarrow x$. 

 \begin{proposition}\label{prop:const_hyp} Suppose $\Omega \subset \Cb^d$ is a bounded convex domain with $C^{1,\epsilon}$ boundary and $H \leq \Aut(\Omega)$ is a subgroup. If $\Lc(\Omega; H)$ intersects at least two different closed complex faces of $\partial \Omega$, then $H$ contains a hyperbolic element. 
 \end{proposition}
 
 This is essentially the proof of Theorem 7.1 in~\cite{Z2017}.
 
 \begin{proof}
 Suppose that $x, y \in \Lc(\Omega;H)$ and $T_x^{\Cb} \partial \Omega \neq T_y^{\Cb} \partial \Omega$. Then there exists sequences $\phi_m, \psi_n \in H$ and $z,w \in \Omega$ such that $\phi_m(z) \rightarrow x$ and $\psi_n(w) \rightarrow y$. If one of the $\phi_m$ or $\psi_n$ is hyperbolic, then there is nothing to show. So suppose that every $\phi_m$ and $\psi_n$ is non-hyperbolic. 
 
 Then pick $U$ a neighborhood of $T_{x}^{\Cb} \partial \Omega \cap \partial \Omega$ in $\overline{\Omega}$ and $V$ a neighborhood of $T_{y}^{\Cb} \partial \Omega \cap \partial \Omega$ in $\overline{\Omega}$ such that $\overline{V} \cap \overline{U} \neq \emptyset$. By Proposition~\ref{prop:non_hyp_attracting} there exists $N \geq 0$ such that 
\begin{align*}
\phi_m^{\pm 1}(\Omega \setminus U) \subset U\text{ and } \psi_n^{\pm 1}(\Omega \setminus V) \subset V
\end{align*}
for all $m,n \geq N$. Next consider the elements $h_{m,n} = \phi_m \psi_n$. Then for $m,n$ large we have that 
\begin{align*}
h_{m,n}( \Omega \setminus V) \subset U \text{ and }  h_{m,n}^{-1}( \Omega \setminus U) \subset V.
\end{align*}
Thus Proposition~\ref{prop:non_hyp_attracting} implies that $h_{m,n}$ must be hyperbolic for $m,n$ large. 
 
 \end{proof}

\section{More on hyperbolic elements}\label{sec:hyp_elem_convex}

In this section we establish a number of new results about hyperbolic elements in $\Aut(\Omega)$. 

 \subsection{Stability of hyperbolic elements}
 
  \begin{proposition}\label{prop:stability_convex} Suppose $\Omega \subset \Cb^d$ is a bounded convex domain with $C^{1,\epsilon}$ boundary. If $h \in \Aut(\Omega)$ is a hyperbolic element, then there exists a neighborhood $\Oc$ of $h$ in $\Aut(\Omega)$ such that every $h^\prime \in \Oc$ is also hyperbolic.  \end{proposition}
  
 \begin{proof}Suppose for a contradiction that there exists $h_n \rightarrow h$ such that each $h_n$ is non-hyperbolic. Now fix some $z_0 \in \Omega$. Then 
 \begin{align*}
 \lim_{m \rightarrow \pm \infty} d_{\Euc}(h^mz_0, H_h^\pm) = 0.
 \end{align*}
  Further, 
  \begin{align*}
  \lim_{n \rightarrow \infty} h_n^m(z_0) = h^m(z_0)
  \end{align*}
  for every $m \in \Nb$. So we can select $m_n \rightarrow \infty$ such that 
   \begin{align*}
 \lim_{n \rightarrow \infty} d_{\Euc}(h_n^{\pm m_n}z_0, H_h^\pm) = 0.
 \end{align*}
 But then by Proposition~\ref{prop:non_hyp_attracting}, the elements $h_n^{m_n}$ must be hyperbolic when $n$ is sufficiently large. Which implies that $h_n$ is hyperbolic when $n$ is sufficiently large. So we have a contradiction. 
\end{proof}

 \subsection{North/South Dynamics}
 
 \begin{proposition}\label{prop:NS} Suppose $\Omega$ is a convex domain with $C^{1,\epsilon}$ boundary and $h \in \Aut(\Omega)$ is a hyperbolic element. If $U$ is a neighborhood of $H_{h}^+ \cap \partial \Omega$ in $\overline{\Omega}$ and $V$ is a neighborhood of $H_{h}^+ \cap \partial \Omega$ in $\overline{\Omega}$, then there exists some $N >0$ such that 
 \begin{align*}
h^n (\Omega \setminus V) \subset U \text{ and } h^{-n} (\Omega \setminus U) \subset V
 \end{align*}
 for all $n \geq N$.
 \end{proposition}
 
 \begin{proof} Since $h^{-1}$ is also hyperbolic, it is enough to prove that there exists some $N > 0$ such that 
  \begin{align*}
h^n (\Omega \setminus V) \subset U
  \end{align*}
 for all $n \geq N$. Suppose not, then there exists $n_k \rightarrow \infty$ and a sequence $q_k \in \Omega \setminus V$ such that $h^{n_k}(q_k) \notin U$. By passing to a subsequence we can suppose that $q_k \rightarrow y_1$ and $h^{n_k}(q_k) \rightarrow y_2$.

 Fix some $z_0 \in \Omega$ and a sequence $m_k \rightarrow-\infty$ such that $m_k+n_k \rightarrow -\infty$. Then let $p_k = h^{m_k}(z_0)$. Then 
\begin{align*}
\lim_{k \rightarrow \infty} d_{\Euc}\left(p_k, H^-_{h}\right) = \lim_{k \rightarrow \infty} d_{\Euc}\left(h^{m_k}(z_0), H^-_{h}\right) = 0
 \end{align*}
 and
 \begin{align*}
\lim_{k \rightarrow \infty} d_{\Euc}\left(h^{n_k} p_k, H^-_{h}\right) = \lim_{k \rightarrow \infty} d_{\Euc}\left(h^{m_k+n_k}(z_0), H^-_{h}\right) = 0.
 \end{align*}

 Next let $\sigma_k:[a_k,b_k] \rightarrow \Omega$ be a sequence of geodesics with $\sigma_k(a_k)=q_k$ and $\sigma_k(b_k) = p_k$. Since $q_k \in \Omega \setminus V$, Theorem~\ref{thm:visible}  implies that we can pass to a subsequence and reparametrize each $\sigma_k$ so that $\sigma_k$ converges to a geodesic $\sigma$. Next consider the geodesics $\sigma_k^{(1)} = h^{n_k}\sigma_k|_{[a_k,0]}$ and $\sigma_k^{(2)} = h^{n_k}\sigma_k|_{[0,b_k]}$. Since $\sigma_k(0) \rightarrow \sigma(0)$, Proposition~\ref{prop:finite_dist} implies that
 \begin{align*}
\lim_{k \rightarrow \infty} d_{\Euc}\left( h^{n_k}\sigma_k(0), H^+_{h}\right) =0.
 \end{align*}
 Further, $\sigma_k^{(1)}(a_k) = h^{n_k}(q_k) \rightarrow y_2$ and
  \begin{align*}
\lim_{k \rightarrow \infty} d_{\Euc}\left( \sigma_k^{(2)}(b_k) , H^-_{h}\right) =\lim_{k \rightarrow \infty} d_{\Euc}\left( h^{n_k} p_k, H^-_{h}\right)=0.
 \end{align*}
 
 So after passing to a subsequence Theorem~\ref{thm:visible} implies that there exists $\alpha_k \in [a_k,0]$ and $\beta_k \in [0,b_k]$ such that 
 the geodesics $t \rightarrow \sigma_k^{(1)}(t+\alpha_k)$ and $t \rightarrow \sigma_k^{(2)}(t+\beta_k)$ converge locally uniformly to geodesics $\sigma^{(1)}$ and $\sigma^{(2)}$ respectively. Since $\Aut(\Omega)$ acts properly on $\Omega$ we must have that $\alpha_k \rightarrow -\infty$ and $\beta_k \rightarrow \infty$. But then 
 \begin{align*}
\infty> K_\Omega\left(\sigma^{(1)}(0), \sigma^{(2)}(0) \right) & = \lim_{k \rightarrow \infty} K_\Omega\left(\sigma^{(1)}_k(\alpha_k), \sigma^{(2)}_k(\beta_k) \right) \\
& = \lim_{k \rightarrow \infty} K_\Omega\left(h^{n_k}\sigma_k(\alpha_k), h^{n_k}\sigma_k(\beta_k) \right)\\
& = \lim_{k \rightarrow \infty} K_\Omega\left(\sigma_k(\alpha_k), \sigma_k(\beta_k) \right) = \lim_{k \rightarrow \infty} \beta_k - \alpha_k = \infty
\end{align*}
so we have a contradiction. 
 
 \end{proof}
 
 \subsection{Applications of North/South Dynamics}
 
\begin{proposition}\label{prop:PP1} Suppose $\Omega$ is a bounded convex domain with $C^{1,\epsilon}$ boundary. If $h_1, h_2 \in \Aut(\Omega)$ are hyperbolic elements and 
\begin{align*}
\{ H_{h_1}^+, H_{h_1}^-\} \cap \{ H_{h_2}^+, H_{h_2}^-\} = \emptyset,
\end{align*}
then there exists $n,m >0$ such that the elements $h_1^m, h_2^n$ generate a free group.
\end{proposition}

\begin{proof}
This follows from Proposition~\ref{prop:NS} and the well known ``ping-pong lemma,'' see for instance~\cite[Section II.B]{P2000}.
\end{proof}

 \begin{proposition}\label{prop:PP2} Suppose that $\Omega \subset \Cb^d$ is a bounded convex domain with $C^{1,\epsilon}$ boundary. If $h_1, h_2 \in \Aut(\Omega)$ are hyperbolic elements,
 \begin{align*}
 \{ H_{h_1}^+, H_{h_1}^-\} \cap  \{ H_{h_2}^+, H_{h_2}^-\}  = \emptyset,
 \end{align*}
$V_1$ is a neighborhood of $H_{h_1}^+ \cap \partial \Omega$ in $\overline{\Omega}$, and $V_2$ is a neighborhood of $H_{h_2}^+ \cap \partial \Omega$ in $\overline{\Omega}$, then there exists $m, n > 0$ such that $h=h_1^m h_2^{-n} \in \Aut(\Omega)$ is a hyperbolic element with $H^+_h \cap \partial \Omega \subset V_1$ and $H^-_h \cap \partial \Omega \subset V_2$.
\end{proposition}
 
 \begin{proof} Fix some $z_0 \in \Omega$. Using Proposition~\ref{prop:NS} we can find $m_k, n_k \rightarrow \infty$ such that if $g_k = h_1^{m_k} h_2^{-n_k}$, then
 \begin{align*}
 \lim_{k \rightarrow \infty} d_{\Euc}(g_k(z_0), H_{h_1}^+) = 0
 \end{align*}
 and
  \begin{align*}
 \lim_{k \rightarrow \infty} d_{\Euc}(g_k^{-1}(z_0), H_{h_2}^+) = 0.
 \end{align*}
 Since $H_{h_1}^+ \neq H_{h_2}^+$, Proposition~\ref{prop:non_hyp_attracting} implies that $g_k$ is hyperbolic for large  $k$. Further, Proposition~\ref{prop:cont_of_att_hyp} implies that
 \begin{align*}
 H_{g_k}^+ \rightarrow H_{h_1}^+
 \end{align*}
 and
  \begin{align*}
 H_{g_k}^- \rightarrow H_{h_2}^+.
 \end{align*}
 So we can pick $k$ such that $H^+_{g_k} \cap \partial \Omega \subset V_1$ and $H^-_{g_k} \cap \partial \Omega \subset V_2$.
  \end{proof}
  
  \subsection{Complex affine disks in the boundary}

A subset $A \subset \Cb^d$ is called an \emph{complex affine disk} if there exists a non-constant complex affine map $\ell: \Cb \rightarrow \Cb^d$ such that $\ell(\Db) = A$.

  \begin{proposition}\label{prop:cplx_disks} Suppose $\Omega \subset \Cb^d$ is a bounded convex domain with $C^{1,\epsilon}$ boundary. Assume $h \in \Aut(\Omega)$ is a hyperbolic element and $x^+_h \in \partial \Omega$ is the boundary point in Theorem~\ref{thm:normal_line_shadowing}. Then there does not exists a complex affine disk $A \subset \partial \Omega$ with $x^+_h \in A$.   \end{proposition}
  
This result follows from the proof of~\cite[Theorem 9.1]{Z2017}, but in this subsection we will provide a different argument.  

\begin{proof}
Fix some $r > 0$ such that $x+r\cdot {\bf n}_{\Omega}(x) \in \Omega$ for all $x \in \partial \Omega$. Then for $x \in \partial \Omega$, define the curve 
\begin{align*}
\sigma_x(t) = x + r e^{-2t} {\bf n}_\Omega(x).
\end{align*}
By Proposition~\ref{prop:normal_lines} there exists some $\kappa > 0$ such that each $\sigma_x$ is an $(1,\kappa)$-almost-geodesic.

Suppose for a contradiction that there exists a complex affine disk $A \subset \partial \Omega$ with $x^+_h \in A$. Let $L$ be the complex affine line containing $A$ and let $\Oc$ denote the interior of $L \cap \partial \Omega$ in $L$. Now by~\cite[Proposition 4.6]{Z2014b} if $y \in \Oc$, then
\begin{align*}
K_\Omega^{\Haus}\Big(\sigma_{y}, \sigma_{x_h^+}\Big) < \infty.
\end{align*}

\noindent \textbf{Claim 1:} There exists some $M_1 > 0$ such that 
\begin{align*}
K_\Omega\Big(\sigma_{y}, \sigma_{x_h^+}(t) \Big) \leq M_1
\end{align*}
for all $y \in \Oc$ and $t \geq 0$. 

\begin{proof}[Proof of Claim 1:] If not we can find $y_n \in \Oc$ and $t_n \geq 0$ such that 
\begin{align*}
K_\Omega\Big(\sigma_{y_n}, \sigma_{x_h^+}(t_n) \Big) > n.
\end{align*}
Since 
\begin{align*}
\sup_{x_1, x_2 \in \partial\Omega} K_\Omega(\sigma_{x_1}(0), \sigma_{x_2}(0) ) < \infty
\end{align*}
we must have that $t_n \rightarrow \infty$. Now by Theorem~\ref{thm:normal_line_shadowing} there exists a sequence $m_n$ such that 
\begin{align*}
\sup_{n \in \Nb} K_\Omega\left( h^{m_n} \sigma_{x_h^+}(0),  \sigma_{x_h^+}(t_n) \right) < \infty.
\end{align*}
So by passing to a subsequence we can assume that $h^{-m_n} \sigma_{x_h^+}(t_n)$ converges to some $z_0\in \Omega$. Since $t_n \rightarrow \infty$ we must have $m_n \rightarrow \infty$. 

Next consider the curves $\gamma_n = h^{-m_n} \sigma_{y_n}$. Then
\begin{align*}
\lim_{n \rightarrow \infty} d_{\Euc}(\gamma_n(0) , H^-_h) = \lim_{n \rightarrow \infty} d_{\Euc}\left(  h^{-m_n}\sigma_{y_n}(0), H^-_h\right) = 0.
\end{align*}
Further 
\begin{align*}
K_\Omega^{\Haus}\Big(\gamma_n, & \sigma_{x_h^+}\Big) = K_\Omega^{\Haus}\Big(\sigma_{y_n}, h^{m_n}\sigma_{x_h^+}\Big) \\
& \leq K_\Omega^{\Haus}\Big(\sigma_{y_n}, \sigma_{x_h^+}\Big) + K_\Omega^{\Haus}\Big(\sigma_{x_h^+}, h^{m_n}\sigma_{x_h^+}\Big) < \infty
\end{align*}
and so by Proposition~\ref{prop:finite_dist}
\begin{align*}
\lim_{t \rightarrow \infty} d_{\Euc}(\gamma_n(t), H^+_h) = 0.
\end{align*}
Then by Theorem~\ref{thm:visible} we can pass to a subsequence and find some $T_n$ such that the almost geodesics $t \rightarrow \gamma_n(t+T_n)$ converge to locally uniformly to an almost geodesic $\gamma : \Rb \rightarrow \infty$. But then 
\begin{align*}
\infty = \lim_{n \rightarrow \infty} K_\Omega\Big(\sigma_{y_n}, \sigma_{x_h^+}(t_n) \Big) 
&\leq \lim_{n \rightarrow \infty} K_\Omega\Big(\sigma_{y_n}(T_n), \sigma_{x_h^+}(t_n) \Big)\\
& = \lim_{n \rightarrow \infty} K_\Omega\Big(\gamma_n(T_n), h^{-m_n}\sigma_{x_h^+}(t_n) \Big) \\
&= K_\Omega(\gamma(0), z) < \infty
\end{align*}
so we have a contradiction. \end{proof}

Next define 
\begin{align*}
M_2:=\sup_{x_1, x_2 \in \partial\Omega} K_\Omega(\sigma_{x_1}(0), \sigma_{x_2}(0) ).
\end{align*}

\noindent \textbf{Claim 2:} $K_\Omega\Big(\sigma_{y}(t), \sigma_{x_h^+}(t) \Big) \leq 2M_1+M_2+3\kappa$
for all $t >0$ and $y \in \Oc$.

\begin{proof}[Proof of Claim 2:] Fix $t > 0$ and $y \in \Oc$. Then there exists some $s > 0$ such that 
\begin{align*}
K_\Omega\Big(\sigma_{y}(s), \sigma_{x_h^+}(t) \Big) \leq M_1.
\end{align*}
Since $\sigma_y$ and $\sigma_{x_h^+}$ are $(1,\kappa)$-almost-geodesics we have 
\begin{align*}
K_\Omega\left( \sigma_y(s), \sigma_{x^+_h}(t)\right) 
&\geq \abs{ K_\Omega\left( \sigma_y(s), \sigma_{y}(0)\right) - K_\Omega\left( \sigma_{x^+_h}(0), \sigma_{x^+_h}(t)\right)}-K_\Omega\left(\sigma_y(0), \sigma_{x^+_h}(0)\right)\\
& \geq \abs{t-s}-2\kappa-M_2.
\end{align*}
So $\abs{t-s} \leq M_1 + M_2 + 2\kappa$. So
\begin{align*}
K_\Omega\left( \sigma_y(t), \sigma_{x^+_h}(t)\right) \leq K_\Omega\left( \sigma_y(t), \sigma_y(s)\right) + K_\Omega\left(\sigma_y(s), \sigma_{x^+_h}(t)\right) \leq 2M_1 + M_2 + 3\kappa.
\end{align*}
\end{proof}

Then taking limits we see that
\begin{align*}
K_\Omega\left( \sigma_y(t), \sigma_{x^+_h}(t)\right) \leq 2M_1 + M_2 + 3\kappa
\end{align*}
for all $t >0$ and $y \in \overline{\Oc}$. But this contradicts Proposition~\ref{prop:finite_dist}. \end{proof}

\subsection{A distance estimate}

\begin{proposition}\label{prop:hyperbolic_QI}  Suppose that $\Omega \subset \Cb^d$ is a bounded convex domain with $C^{1,\epsilon}$ boundary, $h \in \Aut(\Omega)$ is a hyperbolic element, and $z_0 \in \Omega$. Then there exists some $\alpha > 1$ and $\beta >0$ such that 
\begin{align*}
\frac{1}{\alpha} \abs{m-n}-\beta \leq K_\Omega(h^m(z_0),h^{n}(z_0)) \leq \alpha \abs{m-n}+\beta
\end{align*}
for all $m,n \in \Zb$. 
\end{proposition}

The proof of Proposition~\ref{prop:hyperbolic_QI} is essentially the proof of the {\v S}varc-Milnor Lemma given in~\cite[Section IV.B]{P2000}.

\begin{proof} It is enough to show that 
\begin{align*}
\frac{1}{\alpha}m-\beta \leq K_\Omega(h^m(z_0),z_0) \leq \alpha m+\beta
\end{align*}
for all $m \in \Nb$. 

First let
\begin{align*}
\delta:=K_\Omega( h(z_0), z_0 ),
\end{align*}
then 
\begin{align*}
K_\Omega(h^m(z_0), z_0) &\leq \sum_{i=1}^{m} K_\Omega(h^{i-1}(z_0),  h^i(z_0)) = m K_\Omega(z_0, h(z_0)) = \delta m.
\end{align*}

By Proposition~\ref{prop:normal_lines} and Theorem~\ref{thm:normal_line_shadowing} there exists an $(1,\kappa)$-almost-geodesic $\sigma: [0,\infty) \rightarrow \Omega$ such that 
\begin{align*}
R:=K_\Omega^{\Haus}\left( \left\{ h^n z_0  : n \in\Nb\right\}, \sigma([0,\infty))  \right) < \infty.
\end{align*}
Fix $m \in \Nb$. Now there exists some $t_m \in \Rb_{\geq 0}$ such that 
\begin{align*}
K_\Omega(h^m(z_0), \sigma(t_m)) \leq R.
\end{align*}
Then we can pick $0=s_0 < s_1 < \dots < s_N = t_m$ such that $N \leq t_m+1$ and 
\begin{align*}
\abs{s_{i+1}-s_i} \leq 1.
\end{align*}
For each $s_i$ there exists some $m_i \in \Nb$ such that 
\begin{align*}
K_\Omega(h^{m_i}(z_0), \sigma(s_i)) \leq R.
\end{align*}
We can assume that $m_0 =0$ and $m_N = m$. Then
\begin{align*}
K_\Omega(h^{m_i}(z_0), h^{m_{i+1}}(z_0)) \leq 2R +K_\Omega(\sigma(s_i), \sigma(s_{i+1})) \leq 2R+1+\kappa.
\end{align*}
In particular, 
\begin{align*}
K_\Omega(h^{\abs{m_i-m_{i+1}}}(z_0), z_0) \leq 2R +\kappa+1.
\end{align*}
Now since $\Aut(\Omega)$ acts properly on $\Omega$ there exists some $M \geq 0$ such that: if 
\begin{align*}
K_\Omega(h^n(z_0), z_0) \leq 2R+\kappa+1,
\end{align*}
then $\abs{n} \leq M$. Thus $\abs{m_i-m_{i+1}} \leq M$ and
\begin{align*}
m & =  \sum_{i=0}^{N-1} m_{i+1}-m_i \leq \sum_{i=0}^{N-1} \abs{m_{i+1}-m_i} \leq NM \leq M(t_m+1) \\
& \leq M K_{\Omega}(\sigma(t_m), \sigma(0)) + M(\kappa+1) \\
& \leq M K_\Omega(h^m(z_0), z_0) + M R + M K_\Omega(\sigma(0), z_0) + M\kappa +M.
\end{align*}

So there exists $\alpha > 1$ and $\beta >0$ such that 
\begin{align*}
\frac{1}{\alpha} \abs{m-n}-\beta \leq K_\Omega(h^m(z_0),h^{n}(z_0)) \leq \alpha \abs{m-n}+\beta
\end{align*}
for all $m,n \in \Zb$. 
\end{proof}

\subsection{Shadowing an almost-geodesic}\label{subsec:unif_translation}

For the rest of this subsection, suppose that $\Omega \subset \Cb^d$ is a bounded convex domain with $C^{1,\epsilon}$ boundary and $h \in \Aut(\Omega)$ is a hyperbolic element. By Corollary~\ref{cor:alm_geod_shadow} there exists an $(1,\kappa)$-almost-geodesic $\sigma: \Rb \rightarrow \Omega$ such that 
\begin{align*}
R:=K_\Omega^{\Haus}\left( \left\{ h^n z_0  : n \in\Zb\right\}, \sigma(\Rb)  \right) <\infty
\end{align*}
and 
\begin{align*}
\lim_{t \rightarrow \pm \infty} d_{\Euc}(\sigma(t), H^\pm_{h}) = 0.
\end{align*}

Now we define a function $\tau: \Rb \times \Zb \rightarrow \Rb$ by setting $\tau(t,0)=t$ and for $m \in \Zb \setminus\{0\}$ setting
\begin{align*}
\tau(m,t) : = \min\{ s \in \Rb: d(h^m\sigma(t), \sigma(s)) = d(h^m\sigma(t), \sigma(\Rb))\}.
\end{align*}

We will establish the following estimates.

\begin{proposition}\label{prop:translating} With the notation above:
\begin{enumerate}
\item $K_\Omega\Big( \sigma( \tau(m,t)), h^m\sigma(t) \Big) \leq 2R$ for all $m \in \Zb$ and $t \in \Rb$.
\item There exists some $A>1$ and $B >0$ such that 
\begin{align*}
\frac{1}{A} (m-n) - B \leq \tau\left(m,t\right) -\tau\left(n,t\right)  \leq A (m-n) + B 
\end{align*}
for all $m > n$ and $t \in \Rb$. 
\end{enumerate}
\end{proposition}

\begin{proof}
Fix $m \in \Zb$ and $t \in \Rb$. Then there exists $n_t \in \Zb$ such that 
\begin{align*}
K_\Omega( h^{n_t}z_0, \sigma(t)) \leq R.
\end{align*}
and $t_m \in \Rb$ such that 
 \begin{align*}
K_\Omega( h^{m+n_t}z_0, \sigma(t_m)) \leq R.
\end{align*}
Then by definition 
\begin{align*}
K_\Omega\left( \sigma( \tau(m, t)), h^m\sigma(t) \right) \leq K_\Omega(\sigma(t_m), h^{m} \sigma(t)) 
\end{align*}
and
 \begin{align*}
K_\Omega(\sigma(t_m), h^{m} \sigma(t)) \leq K_\Omega(\sigma(t_m), h^{m+n_t}(z_0))+K_\Omega(h^{m+n_t}(z_0), h^{m} \sigma(t))\leq 2R.
\end{align*}
This establishes part (1).

By Proposition~\ref{prop:hyperbolic_QI}, there exists $\alpha>1$ and $\beta > 0$ such that 
\begin{align*}
\frac{1}{\alpha} \abs{m-n}-\beta \leq K_\Omega(h^m(z_0),h^{n}(z_0)) \leq \alpha \abs{m-n}+\beta
\end{align*}
for all $m,n \in \Zb$. Fix $m,n \in \Zb$ and $t \in \Rb$. Then there exists some $n_t \in \Zb$ such that 
\begin{align*}
K_\Omega( h^{n_t}z_0, \sigma(t)) \leq R.
\end{align*}
Then 
\begin{align*}
\abs{K_\Omega( h^{m+n_t}z_0, h^{n+n_t}z_0)-K_\Omega\Big( h^m\sigma(t), h^n\sigma(t) \Big) } \leq 2R
\end{align*}
so by Part (1)
\begin{align*}
\abs{K_\Omega( h^{m+n_t}z_0, h^{n+n_t}z_0)-K_\Omega\Big(  \sigma(\tau(m, t)), \sigma(\tau(n, t)) \Big) } \leq 6R.
\end{align*}
Thus 
\begin{align*}
\frac{1}{\alpha} \abs{m-n} - \beta - 6R \leq K_\Omega\Big(  \sigma(\tau(m,t)), \sigma(\tau(n,t)) \Big) \leq \alpha \abs{m-n} + \beta + 6R.
\end{align*}
Since $\sigma$ is a $(1,\kappa)$-almost-geodesic, this implies that 
\begin{equation}
\label{eq:est_tau}
\frac{1}{\alpha} \abs{m-n} - \beta - 6R - \kappa \leq \abs{ \tau(m, t)-  \tau(n, t)} \leq \alpha \abs{m-n} + \beta + 6R + \kappa.
\end{equation}
So to establish Part (2), we just need to show that there exists some $m_0$ such that $\tau(m,t) - \tau(n,t) > 0$ for all $n \in \Zb$, $m \geq m_0+n$, and $t \in \Rb$.

Equation~\eqref{eq:est_tau} implies that there exist some $C>0$ such that
\begin{align*}
\abs{ \tau(k,t) - \tau(k+1,t)} \leq C
\end{align*}
for all $t \in \Rb$ and $k \in \Zb$. Further
\begin{align*}
\lim_{k \rightarrow \infty} \tau(k,t) = \infty.
\end{align*}
So if $m > n$ and $\tau(m,t) - \tau(n,t) < 0$, then there exists some $M \geq m$ such that 
\begin{align*}
\tau(n,t) \leq \tau(M,t)  \leq C+\tau(n,t).
\end{align*}
But then 
\begin{align*}
m-n \leq M-n \leq \alpha \left( C + \beta + 6 R + \kappa \right).
\end{align*}
\end{proof}

\section{Constructing one-parameter subgroups}\label{sec:frankel_refined}

In the late 1980's, Frankel~\cite{F1989,F1991} developed a method to construct one-parameter subgroups of $\Aut(\Omega)$ when $\Omega$ is convex and $\Aut(\Omega)$ is non-compact. In particular, his method implies the following. 

\begin{proposition}\label{prop:frankel} Suppose $\Omega$ is a convex domain with $C^1$ boundary. If $\Aut(\Omega)$ is non-compact, then $\Aut(\Omega)$ contains an one-parameter group $u_t$ of automorphisms. 
\end{proposition}

Here is a sketch of Frankel's argument: suppose that $\varphi_n \rightarrow \infty$ in $\Aut(\Omega)$. One can then pass to a subsequence and find certain affine automorphisms  $A_n : \Cb^d \rightarrow \Cb^d$ such that $\Omega_n:=A_n \Omega$ converges in the local Hausdorff topology to a convex domain $\wh{\Omega}$. By selecting the affine maps carefully one can also show that the maps $A_n \varphi_n : \Omega \rightarrow \Omega_n$ converge to a biholomorphism $\Phi: \Omega \rightarrow \wh{\Omega}$. Finally, since $\partial \Omega$ is $C^1$, it turns out that $\wh{\Omega}$ contains a real line $z_0 + \Rb u$. Then since $\wh{\Omega}$ is convex and open, $z+\Rb u \subset \wh{\Omega}$ for all $z \in \wh{\Omega}$. So $\Aut\left(\wh{\Omega}\right)$ contains the one-parameter group
\begin{align*}
\wh{u}_t(z) = z + tu.
\end{align*}
So $\Aut(\Omega)$ contains the one-parameter group $u_t = \Phi^{-1} \wh{u}_t \Phi$. 

One problem with Frankel's method is that there is no obvious connection between the initial sequence $\varphi_n$ and the resulting one-parameter group $u_t$. In this section we will apply Frankel's method to a sequence $\varphi_n = h^n$ where $h$ is hyperbolic and use the properties of hyperbolic elements established in Section~\ref{sec:hyp_elem_convex} to prove the following.

\begin{theorem}\label{thm:rescaling}
Suppose $\Omega$ is a bounded convex domain with $C^{1,\epsilon}$ boundary. If $h \in \Aut(\Omega)$ is a hyperbolic element, then $\Aut(\Omega)$ contains an one-parameter group $u_t$ of parabolic automorphisms such that 
\begin{align*}
\limsup_{n \rightarrow \infty} K_\Omega( u_t h^n z_0, h^n z_0) < \infty
\end{align*}
for all $z \in \Omega$ and $t \in \Rb$. In particular, $H^+_{u_t} = H^+_h$. 
\end{theorem}

Before starting the proof of Theorem~\ref{thm:rescaling} we need to recall some basic facts about the local Hausdorff topology. 

First let $\Xb_d$ denote the set of all convex domains in $\Cb^d$ which do not contain an affine line. By a theorem of Barth, see Theorem~\ref{thm:barth} above, $\Xb_d$ consists of exactly the convex domains where the Kobayashi pseudo-distance is non-degenerate.

For $R \geq 0$, let $B_R(z)$ denote the open ball of radius $R$ centered at $z$ with respect to the Euclidean distance. We then introduce a topology on $\Xb_d$ by saying that a sequence $\Omega_n \in \Xb_d$ converges to $\Omega \in \Xb_d$ if there exists some $R_0 \geq 0$ such that 
\begin{align*}
\lim_{n \rightarrow \infty} d_{\Euc}^{\Haus} \left( \Omega_n \cap B_R(0), \Omega \cap B_R(0) \right) = 0.
\end{align*}
for all $R \geq R_0$. 

The Kobayashi distance behaves nicely with respect to this notion of convergence. 

\begin{theorem}\cite[Theorem 4.1]{Z2014} Suppose a sequence $\Omega_n \in \Xb_d$ converges to $\Omega \in \Xb_d$. Then $K_{\Omega_n} \rightarrow K_\Omega$ uniformly on compact subsets of $\Omega \times \Omega$.
\end{theorem}

We next let $\Xb_{d,0}$ denote the set of all pairs $(\Omega, z)$ where $\Omega \in \Xb_d$ and $z \in \Omega$. This set also has a topology where $(\Omega_n, z_n) \rightarrow (\Omega, z)$ if and only if $\Omega_n \rightarrow \Omega$ and $z_n \rightarrow z$.

Next let $e_1,\dots, e_d$ denote the standard basis of $\Cb^d$. Then let $\Kb_{d,0} \subset \Xb_{d,0}$ consist of all elements $(\Omega,0)$ where 
\begin{align*}
\Db e_1 \cup \dots \cup \Db e_d \subset \Omega
\end{align*}
and 
\begin{align*}
\Omega \cap \Big( e_i + \Span_{\Cb}\{ e_{i+1}, \dots, e_d\} \Big)= \emptyset \text{ for every $1 \leq i \leq d$.}
\end{align*}
With this notation we have the following.

\begin{theorem}\cite[Theorem 2.5]{Z2016} The subset $\Kb_{d,0} \subset \Xb_{d,0}$ is compact. Moreover, if $\Aff(\Cb^d)$ is the affine automorphism group of $\Cb^d$, then $\Aff(\Cb^d) \cdot \Kb_{d,0} = \Xb_{d,0}$. 
\end{theorem}

\begin{remark} Frankel has constructed a slightly different compact set $K \subset \Xb_{d,0}$ such that $\Aff(\Cb^d) \cdot K = \Xb_{d,0}$~\cite{F1991}. \end{remark}

We are now ready to prove Theorem~\ref{thm:rescaling}. 

\begin{proof}[Proof of Theorem~\ref{thm:rescaling}]
For notational convenience, we will construct a one-parameter group $u_t$ of parabolic automorphisms such that 
\begin{align*}
\lim_{n \rightarrow \infty} K_\Omega( u_t h^{-n} z_0, h^{-n} z_0) < \infty
\end{align*}
for all $z \in \Omega$ and $t \in \Rb$. Since $h^{-1}$ is also hyperbolic, this will imply the theorem. 

Let $x^+_h \in \partial \Omega$ be as in Theorem~\ref{thm:normal_line_shadowing}. By translating, rotating, and scaling $\Omega$ we may assume that $x^+_h = 0$,
\begin{align*}
T_{x^+_h} \partial \Omega = \{ (z_1, \dots, z_d) \in \Cb^d: \Imaginary(z_1)=0\},
\end{align*}
and $ie_1 \in \Omega$. Then by  Theorem~\ref{thm:normal_line_shadowing} 
\begin{align*}
R:=K_\Omega^{\Haus}\left( \left\{ h^n (ie_1) : n \in\Nb\right\}, (0,1] \cdot i e_1  \right) < \infty.
\end{align*}

Define $\sigma: [0,\infty) \rightarrow \Omega$ by 
\begin{align*}
\sigma(t) = e^{-2t} ie_1.
\end{align*}
Then by Proposition~\ref{prop:normal_lines} there exists some $\kappa > 1$ such that $\sigma$ is an $(1,\kappa)$-almost-geodesic.

Then pick a sequence $t_n \rightarrow \infty$ and consider the points $p_n = \sigma(t_n)$. For each $n$, we will construct an affine map $A_n$ such that $A_n(\Omega, p_n) \in \Kb_{d,0}$. To do this, we begin by selecting points $x_1^{(n)}, \dots, x_{d}^{(n)} \in \partial \Omega$ using the following procedure. First let $x_1^{(n)}= 0$. Now supposing that $x_1^{(n)}, \dots, x_k^{(n)}$ have already been selected, let $P_k$ denote the maximal dimensional complex affine subspace through $p_n$ which is orthogonal to the lines
\begin{align*}
\overline{x_1^{(n)}p_n}, \ \dots \ ,  \overline{x_k^{(n)}p_n}
\end{align*}
then let $x_{k+1}^{(n)}$ be a point in $\partial \Omega \cap P_k$ closest to $p_n$. 

Next for each $n$, let $T_n : \Cb \rightarrow \Cb$ denote the translation $T_n(z) = z-p_n$ and let $U_n$ denote the unitary matrix such that 
\begin{align*}
U_n\left(T_n\left(x_n^{(i)}\right)\right) = \norm{ p_n - x_n^{(i)}}e_i \text{ for } i=1,\dots, d.
\end{align*}
Then let $\Lambda_n$ denote the diagonal matrix with 
\begin{align*}
\Lambda_n( e_i) = \frac{1}{\norm{ p_n - x_n^{(i)}}} e_i \text{ for } i=1,\dots, d.
\end{align*}
Finally let $A_n = \Lambda_n U_n T_n$. Then by construction $A_n(\Omega)$ contains 
\begin{align*}
\Db e_1 \cup \dots \cup \Db e_d .
\end{align*}
Since each  $\Omega \cap P_k$ is a convex set with $C^1$ boundary we also have
\begin{align*}
A_n(\Omega) \cap \Big( e_i + \Span_{\Cb}\{ e_{i+1}, \dots, e_d\} \Big)= \emptyset
\end{align*}
for $1 \leq i \leq d$. Further, $A_n(p_n) =0$. So $A_n(\Omega, p_n) \in \Kb_{d,0}$. Since $\Kb_{d,0}$ is compact, we can pass to a subsequence such that $A_n(\Omega)$ converges in the local Hausdorff topology to a convex domain $\wh{\Omega}$ in $\Xb_d$.

Since $p_n = \sigma(t_n)$, there exists a sequence $m_n \in \Nb$ such that 
\begin{align*}
K_\Omega( h^{m_n} (ie_1), p_n) \leq R.
\end{align*}
Then consider the maps $\Phi_n = A_n h^{m_n} : \Omega \rightarrow A_n \Omega$. We claim that after passing to a subsequence $\Phi_n$ converges locally uniformly to a biholomorphism $\Phi:\Omega \rightarrow \wh{\Omega}$. Since
\begin{align*}
K_{A_n\Omega}(0, \Phi_n(ie_1)) =  K_{\Omega}(A_n^{-1}(0), h^{m_n} (ie_1)) = K_\Omega(p_n, h^{m_n}(ie_1)) \leq R
\end{align*}
and $K_{A_n\Omega}$ converges locally uniformly to $K_{\wh{\Omega}}$, we can use to the Arzel{\' a}-Ascoli theorem to pass to a subsequence such that $\Phi_n$ converges locally uniformly to an isometry $\Phi: (\Omega,K_\Omega) \rightarrow (\wh{\Omega}, K_{\wh{\Omega}})$. Then, since locally uniform limits of holomorphic maps are holomorphic, we see that $\Phi$ is a holomorphic. Since $(\wh{\Omega}, K_{\wh{\Omega}})$ is a metric space, we see that $\Phi$ one-to-one and onto. So $\Phi$ is a biholomorphism, see~\cite[p. 86]{N1971}.

We now show that $\Aut\left(\wh{\Omega}\right)$ contains a one-parameter subgroup, but first an observation. \newline 

\noindent \textbf{Claim 1:} $\{ ze_1 : \Imaginary(z) < 1\} \subset \wh{\Omega}$. 

\begin{proof}[Proof of Claim 1]
For $\epsilon, \delta > 0$ define 
\begin{align*}
C(\epsilon,\delta):=\{ ze_1 : \abs{z} < \delta \text{ and } \Imaginary(z) > \epsilon \abs{\Real(z)} \}.
\end{align*}
Fix $\epsilon > 0$. Since $\partial \Omega$ is $C^1$, $0 \in \partial \Omega$, and ${\bf n}_\Omega(0) = ie_1$: there exists some $\delta >0$ such that $C(\epsilon,\delta) \subset \Omega$. Then 
\begin{align*}
A_n \Big( C(\epsilon,\delta)\Big) =\{ ze_1 : \abs{z-1} < e^{2t_n}\delta \text{ and } \Imaginary(z) < 1-\epsilon \abs{\Real(z)} \}.
\end{align*}
Since $A_n \Omega \rightarrow \wh{\Omega}$ and $e^{2t_n} \rightarrow \infty$ we then have
\begin{align*}
\{ ze_1 : \Imaginary(z) < 1-\epsilon \abs{\Real(z)} \} \subset \wh{\Omega}.
\end{align*}
Since $\epsilon > 0$ was arbitrary, we then see that 
\begin{equation*}
\{ ze_1 : \Imaginary(z) < 1\} \subset \wh{\Omega}. \qedhere
\end{equation*}
\end{proof}

The above claim implies that $\wh{\Omega}$ contains the real line $\Rb e_1$. Since $\wh{\Omega}$ is open and convex, we have
\begin{align*}
z + \Rb e_1 \subset \wh{\Omega}
\end{align*}
for all $z \in \wh{\Omega}$. Thus $\Aut\left(\wh{\Omega}\right)$ contains the one-parameter group $\wh{u}_t$ defined by
\begin{align*}
\wh{u}_t (z) = z + te_1.
\end{align*}

\noindent \textbf{Claim 2:} $\lim_{s \rightarrow \infty} K_{\wh{\Omega}}\left( \wh{u}_t(1-e^{2s})ie_1, (1-e^{2s})ie_1 \right) = 0$.

\begin{proof}[Proof of Claim 2] Let $\Hc = \{ z \in \Cb : \Imaginary(z) > 0\}$. Then since
\begin{equation*}
\{ ze_1 : \Imaginary(z) < 1\} \subset \wh{\Omega},
\end{equation*}
the distance decreasing property of the Kobayashi metric implies that 
\begin{align*}
K_{\wh{\Omega}}\left( \wh{u}_t(1-e^{2s})ie_1, (1-e^{2s})ie_1 \right) \leq K_{\Hc}(t+e^{2s}i, e^{2s}i).
\end{align*}
Further the map $z \rightarrow e^{-2s}z$ is in $\Aut(\Hc)$ and so
\begin{align*}
K_{\Hc}(t+e^{2s}i, e^{2s}i) = K_{\Hc}( te^{-2s}+i, i)
\end{align*}
which clearly converges to 0 as $s \rightarrow \infty$.
\end{proof}

\noindent \textbf{Claim 3:} Let $\wh{h} = \Phi \circ h \circ \Phi^{-1} \in \Aut(\wh{\Omega})$. Then there exists a sequence $\tau_m \rightarrow -\infty$ such that
\begin{align*}
\limsup_{m \rightarrow \infty} K_{\wh{\Omega}}\left( \wh{h}^{-m} (0), (1-e^{-2\tau_m}) ie_1 \right) \leq 2R.
\end{align*}

\begin{proof}[Proof of Claim 3]
Using (the proof of) Corollary~\ref{cor:alm_geod_shadow} we can extend $\sigma$ to a $(1,\kappa_1)$-almost-geodesic $\sigma_1: \Rb \rightarrow \Omega$ such that 
\begin{align*}
K_\Omega^{\Haus}\left( \left\{ h^n z_0  : n \in\Zb\right\}, \sigma_1(\Rb)  \right) <\infty
\end{align*}
and 
\begin{align*}
\lim_{t \rightarrow \pm \infty} d_{\Euc}(\sigma_1(t), H^\pm_{h}) = 0.
\end{align*}
Then, as in Subsection~\ref{subsec:unif_translation}, define a function $\tau: \Rb \times \Zb \rightarrow \Rb$ by setting $\tau(t,0)=t$ and for $m \in \Zb \setminus\{0\}$ setting
\begin{align*}
\tau(m,t) : = \min\{ s \in \Rb: d(h^m\sigma_1(t), \sigma_1(s)) = d(h^m\sigma_1(t), \sigma_1(\Rb))\}.
\end{align*}
By Proposition~\ref{prop:translating} part (2), there exists some $A > 1$ and $B > 0$ such that
\begin{align*}
-Am - B \leq \tau(-m, t) - t \leq -\frac{1}{A}m +B 
\end{align*}
for $m > 0$ and $t \in \Rb$.

Then by passing to a subsequence we can assume that 
\begin{align*}
\tau_m:=\lim_{n \rightarrow\infty} \tau(-m, t_n)-t_n
\end{align*}
exists for each $m$. Further, $\lim_{m \rightarrow \infty} \tau_m =-\infty$.

Then by Proposition~\ref{prop:translating} part (1)
\begin{align*}
K_{\wh{\Omega}}\left( \wh{h}^{-m} (0), (1-e^{-2\tau_m}) ie_1 \right) 
& \leq \liminf_{n \rightarrow \infty} K_{A_n \Omega_n}\Big( \Phi_n \circ h^m \circ \Phi_n^{-1} (0), (1-e^{-2\tau(-m, t_n)+2t_n}) \cdot ie_1 \Big) \\
& = \liminf_{n \rightarrow \infty} K_{A_n \Omega_n}\left( A_n \circ h^m \circ A_n^{-1} (0), A_n\left(e^{-2\tau(-m, t_n)}\right) \cdot ie_1) \right)\\
& = \liminf_{n \rightarrow \infty} K_{\Omega}\Big( h^m\sigma(t_n), \sigma(\tau(-m, t_n))\Big) \leq 2R.
\end{align*}
\end{proof}

Finally we have 
\begin{align*}
\limsup_{m \rightarrow \infty} & \ K_\Omega( u_t h^{-m} z_0, h^{-m} z_0) \leq 2K_{\wh{\Omega}}(\Phi(z_0), 0) +   \limsup_{m \rightarrow \infty} K_{\wh{\Omega}}\left( \wh{u}_t \wh{h}^{-m} (0), \wh{h}^{-m} (0)\right)\\
&  \leq 2K_{\wh{\Omega}}(\Phi(z_0), 0)+ 4R + \limsup_{m \rightarrow \infty} K_\Omega\left( \wh{u}_t (1-e^{-2\tau_m}) ie_1, (1-e^{-2\tau_m} )ie_1\right) \\
& =2K_{\wh{\Omega}}(\Phi(z_0), 0)+ 4R.
\end{align*}

\end{proof}

\section{Constructing more hyperbolic elements}\label{sec:construct_more_hyp}

In this section we use Theorem~\ref{thm:rescaling} to construct more hyperbolic elements. 

\begin{theorem}\label{thm:construct_hyp} Suppose that $\Omega$ is a bounded convex domain with $C^{1,\epsilon}$ boundary and $\Lc(\Omega)$ intersects at least two different closed complex faces of $\partial \Omega$. Then given any finite list of points $x_1, \dots, x_n$ there exists a hyperbolic element $h \in\Aut(\Omega)$ such that 
\begin{align*}
H^+_{h}, H^-_{h} \notin \left\{ T_{x_1}^{\Cb} \partial \Omega, \dots, T_{x_n}^{\Cb} \partial \Omega\right\}.
\end{align*}
Moreover, if $h_0 \in \Aut(\Omega)$ is any hyperbolic element, then we can assume that $h$ is in the subgroup of $\Aut(\Omega)$ generated by $\{ g h_0g^{-1} : g \in \Aut_0(\Omega) \}$.
\end{theorem}

We begin the proof of Theorem~\ref{thm:construct_hyp} with the following lemma.
  
 \begin{lemma}\label{lem:conj_cont} Suppose $\Omega \subset \Cb^d$ is a bounded convex domain with $C^{1,\epsilon}$ boundary and $h \in \Aut(\Omega)$ is a hyperbolic element. Then the map 
 \begin{align*}
 g \in \Aut(\Omega) \rightarrow H_{ghg^{-1}}^+
 \end{align*}
 is continuous. 
 \end{lemma}
 
 \begin{proof} Suppose that $g_n \rightarrow g$ in $\Aut(\Omega)$. 
 
By Proposition~\ref{prop:normal_lines} and Theorem~\ref{thm:normal_line_shadowing}, there exists an $(1,\kappa)$-almost-geodesic $\sigma: [0,\infty) \rightarrow \Omega$ such that
\begin{align*}
R:=K_\Omega^{\Haus}(\sigma([0,\infty)), \{ h^m(\sigma(0)) : m \in \Nb\} ) < \infty.
\end{align*}
Next let $\sigma_n = g_n \sigma$ and $\sigma_\infty = g \sigma$. Define
\begin{align*}
R_1 := R+ \sup_{n \in \Nb} K_\Omega(g_n^{-1}(\sigma(0)), \sigma(0)),
\end{align*}
then
\begin{align*}
K_\Omega^{\Haus}\left(\sigma_n([0,\infty)), \{ g_nh^mg_n^{-1}(\sigma(0)) :m  \in \Nb\} \right) \leq R_1
\end{align*} 
for all $n \in \Nb$. \newline

\noindent \textbf{Claim:} $\lim_{t,m,n \rightarrow \infty} \left( \sigma_\infty(t) | g_n h^{m} g_n^{-1}(z_0)\right)_{z_0}  = \infty$ for all $z_0 \in \Omega$. 

\begin{proof}[Proof of Claim]
First note that it is enough to show that 
\begin{align*}
\lim_{t,m,n \rightarrow \infty} \left( \sigma_\infty(t) | g_n h^{m} g_n^{-1}(w_0)\right)_{z_0}  = \infty
\end{align*}
for some $w_0, z_0 \in \Omega$.

For all $n,m>0$, there exist some $t_{n,m} \in \Rb$ such that 
\begin{align*}
K_\Omega\Big(\sigma_n(t_{n,m}), g_n h^{m} g_n^{-1}(\sigma(0))\Big) \leq R_1.
\end{align*}

Now fix some $T > 0$. Then pick $t,m,n$ large enough such that $\min\{t,t_{n,m}\} > T$ and $K_\Omega(\sigma_n(s), \sigma_\infty(s)) \leq 1$ for $s \in [0,T]$. Then for $t > T$ we have
\begin{align*}
K_\Omega(\sigma_\infty(t), \sigma_n(t_{n,m})) & \leq (t-T) + (t_{n,m}-T) + K_\Omega(\sigma_\infty(T), \sigma_n(T)) + 2 \kappa \\
& \leq (t-T)+(t_{n,m}-T) + 2\kappa +1.
\end{align*}
Further 
\begin{align*}
K_\Omega(\sigma_\infty(t), \sigma_\infty(0)) \geq t - \kappa 
\end{align*}
and
\begin{align*}
K_\Omega(\sigma_n(t_{n,m}), \sigma_\infty(0)) & \geq K_\Omega(\sigma_n(t_{n,m}), \sigma_n(0))- 1 \geq t_{n,m} - \kappa -1.
\end{align*}
So
\begin{align*}
\left( \sigma_\infty(t) | g_n h^{m} g_n^{-1}(\sigma(0))\right)_{\sigma_\infty(0)} \geq \left( \sigma_\infty(t) | \sigma_n(t_{n,m})\right)_{\sigma_\infty(0)} - R_1 \geq T - 2 \kappa -1 -R_1
\end{align*}
for $t,m,n$ sufficiently large. 

Since $T$ was arbitrary we then have
\begin{align*}
\lim_{t,n,m \rightarrow \infty} \left( \sigma_\infty(t) | g_n h^{m} g_n^{-1}(\sigma(0))\right)_{\sigma_\infty(0)}  = \infty
\end{align*}
which implies the claim.
\end{proof}

Since 
\begin{align*}
\lim_{t \rightarrow \infty} d_{\Euc}(\sigma_\infty(t), H_{g h g^{-1}}^+) = 0,
\end{align*}
Theorem~\ref{thm:GP} then implies that 
\begin{align*}
\lim_{n,m \rightarrow \infty} d_{\Euc}(g_n h^{m} g_n^{-1}(\sigma(0)), H_{g h g^{-1}}^+) = 0.
\end{align*}
So by Proposition~\ref{prop:cont_of_att_hyp}, we see that $H_{g_n h g_n^{-1}}^+ \rightarrow H_{g h g^{-1}}^+$.
 \end{proof}
 
 \begin{proof}[Proof of Theorem~\ref{thm:construct_hyp}]
 By Proposition~\ref{prop:const_hyp}, there exists some hyperbolic element $h_0 \in \Aut(\Omega)$.  Then by Theorem~\ref{thm:rescaling}, there exists a one-parameter group $u^{+}_t$ of parabolic elements such that 
 \begin{align*}
 \lim_{m \rightarrow \infty} K_\Omega( u^+_t h_0^m z_0, h_0^m z_0) < \infty.
 \end{align*}
 Then define 
  \begin{align*}
 h^+_t := u^{+}_t h_0 u^{+}_{-t}
 \end{align*}

Then
 \begin{align*}
K_\Omega( (h^{+}_t)^{m} z_0, h_0^{m} z_0) \leq K_\Omega(u^{+}_{-t} z_0, z_0) + K_\Omega( u^+_t h_0^m z_0, h_0^m z_0)
 \end{align*}
 so
 \begin{align*}
 \limsup_{m \rightarrow \infty}K_\Omega( (h^{+}_t)^{m} z_0, h_0^{m} z_0) < \infty.
 \end{align*}
 Hence by Proposition~\ref{prop:finite_dist}
 \begin{align*}
 H^+_{h^+_t} = H^{+}_{h_0}.
 \end{align*}
We next claim that 
  \begin{align*}
 \lim_{t \rightarrow \infty} H^-_{h^+_t} = H^{+}_{h_0}.
 \end{align*}
By Proposition~\ref{prop:cont_of_att_hyp} it is enough to find some $m_t \in \Nb$ such that 
 \begin{align*}
 \lim_{t \rightarrow \infty} d_{\Euc}( (h^+_t)^{-m_t} z_0 , H^{+}_{h_0}) = 0.
 \end{align*}
Now fix a neighborhood $U$ of $H^+_{h_0} \cap \partial \Omega$ in $\overline{\Omega}$ such that 
 \begin{align*}
 H^-_{h_0} \cap U = \emptyset. 
 \end{align*}
Then by Proposition~\ref{prop:non_hyp_attracting}: if $t \in [0,\infty) \rightarrow z_t$ is any path in $\Omega \setminus U$, then 
 \begin{align*}
\lim_{t \rightarrow \infty} d_{\Euc}(u_t^+(z_t), H^{+}_{h_0}) = 0.
 \end{align*}
Since 
 \begin{align*}
 \lim_{m \rightarrow \infty} d_{\Euc}(h_0^{-m} u_{-t}^+(z_0), H^-_{h_0}) = 0,
 \end{align*}
 for each $t$ we can find $m_t \in \Nb$ such that $h_0^{-m_t} u_{-t}^+ \in \Omega \setminus U$. But then we have 
  \begin{align*}
 \lim_{t \rightarrow \infty} d_{\Euc}( (h^+_t)^{-m_t} z_0 , H^{+}_{h_0}) = 0.
 \end{align*}
And hence
   \begin{align*}
 \lim_{t \rightarrow \infty} H^-_{h^+_t} = H^{+}_{h_0}.
 \end{align*}
 
 Now we can repeat the same construction starting with $h_0^{-1}$ to find a continuous path $h^-_t$ of automorphisms in $\{ g h_0^{-1}g^{-1} : g \in \Aut_0(\Omega)\}$ such that 
  \begin{align*}
 H^+_{h^-_t} = H^{-}_{h_0}
 \end{align*}
 and 
   \begin{align*}
 \lim_{t \rightarrow \infty} H^-_{h^-_t} = H^{-}_{h_0}.
 \end{align*}
 
By Lemma~\ref{lem:conj_cont} the paths $t \rightarrow H^-_{h^+_t}$ and $t \rightarrow H^-_{h^-_t}$ are continuous. So we can pick $t_1, t_2$ such that: if $h_1 = h^+_{t_1}$ and $h_2=h^-_{t_2}$, then 
 \begin{align*}
 \{ H^+_{h_1}, H^-_{h_1} \} \cap  \{ H^+_{h_2}, H^-_{h_2} \} = \emptyset
 \end{align*}
 and 
 \begin{align*}
 \{ H^-_{h_1}, H^-_{h_2} \} \cap \{ T_{x_1}^{\Cb} \partial \Omega, \dots, T_{x_n}^{\Cb} \partial \Omega\}=\emptyset.
 \end{align*}
 So by Proposition~\ref{prop:PP2} there exists a hyperbolic element of the form $h=h_1^{-m} h_2^{n}$ such that 
  \begin{align*}
 \{ H_{h}^+, H_{h}^-\} \cap  \{ T_{x_1}^{\Cb} \partial \Omega, \dots, T_{x_n}^{\Cb} \partial \Omega\}  = \emptyset.
 \end{align*}
 Moreover, by construction $h$ is contained in the subgroup of $\Aut(\Omega)$ generated by $\{ g h_0g^{-1} : g \in \Aut_0(\Omega) \}$. 
 \end{proof}

\section{Proof of Theorem~\ref{thm:main_convex}}\label{sec:proof_convex}

In this section we prove Theorem~\ref{thm:main_convex}. For the rest of this section, suppose that $\Omega$ is a bounded convex domain with $C^{1,\epsilon}$ boundary and $\Lc(\Omega)$ intersects at least two different closed complex faces of $\partial \Omega$.

\begin{lemma}\label{lem:non_compact} With the notation above, $\Aut_0(\Omega)$ is non-compact. \end{lemma}

\begin{proof} By Proposition~\ref{prop:const_hyp} and Theorem~\ref{thm:rescaling}, we know that $\Aut(\Omega)$ contains a one-parameter subgroup of parabolic automorphisms. \end{proof}

\subsection{The connected component of the identity is an almost direct product} 

Let $G^{sol} \leq \Aut_0(\Omega)$ be the solvable radical of $\Aut_0(\Omega)$, that is let $G^{sol}$ be the maximal connected, closed, normal, solvable subgroup of $\Aut_0(\Omega)$. Notice that $G^{sol}$ is also a normal subgroup of $\Aut(\Omega)$. Next let $G^{ss} \leq \Aut_0(\Omega)$ be a connected semisimple subgroup such that $\Aut_0(\Omega) = G^{ss} G^{sol}$ is a Levi-Malcev decomposition of $\Aut_0(\Omega)$.

We next recall a basic fact about solvable Lie groups.

\begin{proposition}\label{prop:decomp_solv} Suppose $S$ is a connected solvable Lie group. Then there exists one-parameter subgroups $S_1, \dots, S_N \leq S$ such that 
\begin{align*}
S = S_1 \cdot S_2 \cdots S_N.
\end{align*}
\end{proposition}

\begin{proof}[Proof Sketch] This is well known, but here is an argument: We induct on the length of the derived series of $S$. Since every connected Abelian Lie group is isomorphic to a $\mathbb{T}^k \times \Rb^\ell$, see for instance~\cite[Corollary 1.103]{K2002}, this is clearly true in the base case. Then for a solvable group $S$, the quotient $S/[S,S]$ is abelian and hence there exists one-parameter subgroups $S_1,\dots, S_k \leq S$ such that $S=S_1\cdots S_k [S,S]$. By induction there exists one-parameter subgroups $S_{k+1}, \dots, S_N \leq [S,S]$ such that $[S,S] = S_{k+1} \cdots S_N$. So $S=S_1\cdots S_N$.
 \end{proof}

\begin{lemma} With the notation above, $G^{sol}$ is compact. In particular, $G^{ss}$ is non-compact. \end{lemma}

\begin{proof} 
Since $\Aut_0(\Omega)$ is non-compact, the ``in particular'' part will follow from the first assertion.  

Suppose that $G^{sol}$ is non-compact, then Proposition~\ref{prop:decomp_solv} implies that $G^{sol}$ contains an element $s$ which is parabolic or hyperbolic. 

First consider the case when $G^{sol}$ contains a hyperbolic element $h_0$. Then by Theorem~\ref{thm:construct_hyp}, there exists a hyperbolic element $h_1 \in \Aut(\Omega)$ such that 
\begin{align*}
\{ H_{h_0}^+, H_{h_0}^-\} \cap \{ H_{h_1}^+, H_{h_1}^-\} = \emptyset.
\end{align*}
Further we can assume that $h_1$ is contained in the group generated by $\{ g h_0 g^{-1} : g \in \Aut_0(\Omega)\}$. Since $G^{sol}$ is normal in $\Aut(\Omega)$, we see that $h_1 \in G^{sol}$. But then by Proposition~\ref{prop:PP1}, $G^{sol}$ contains a free group. So we have a contradiction.

Next consider the case when $G^{sol}$ contains a parabolic element $u$.  By Theorem~\ref{thm:construct_hyp}, $\Aut(\Omega)$ contains a hyperbolic element $h$ such that 
\begin{align*}
H_u^+,H_h^+, H_h^-
\end{align*}
are all distinct. Now the elements $u_{m,n} = h^n u^m h^{-n}$ are contained in $G^{sol}$. Further, Proposition~\ref{prop:non_hyp_attracting} and Proposition~\ref{prop:NS} imply that 
\begin{align*}
\lim_{m,n \rightarrow \infty} d_{\Euc}(u_{m,n}(z_0), H_{h}^+) = 0.
\end{align*}
But then $\Lc(\Omega; G^{sol})$ intersects at least two closed complex faces of $\partial \Omega$. So by Proposition~\ref{prop:const_hyp}, $G^{sol}$ contains a hyperbolic element. So we have a contradiction by the argument above.

\end{proof}

\begin{lemma}\label{lem:torus_convex} With the notation above, $G^{sol}$ is a torus and $G^{sol}$ is the center of $\Aut_0(\Omega)$. \end{lemma}

\begin{proof} This is identical to the proof of Lemma 5.3 in~\cite{Z2017b}. \end{proof}

\begin{lemma} With the notation above, $G^{ss}$ is a normal subgroup in $\Aut(\Omega)$. \end{lemma}

\begin{proof} This is identical to the proof of Lemma 5.5 in~\cite{Z2017b}. \end{proof}

As described in Section~\ref{sec:basic_properties}, there exists closed subgroups $G_1, \dots, G_p \leq G$ such that $G^{ss}$ is the almost direct product of $G_1, \dots, G_p$. Then define subgroups of $G^{ss}$:
\begin{align*}
N_0 = \prod\{ G_i : G_i \text{ is compact} \} 
\end{align*}
and 
\begin{align*}
G= \prod\{ G_i : G_i \text{ is non-compact} \} .
\end{align*}
Then let $N = N_0 G^{sol}$. Since $G^{ss}$ is a normal subgroup in $\Aut(\Omega)$, $N$ and $G$ are also normal subgroups in $\Aut(\Omega)$ (for details see the proof of~\cite[Lemma 5.5]{Z2017b}). Further, since $G^{sol}$ is in the center of $\Aut_0(\Omega)$ we have the following.

\begin{lemma} With the notation above, $\Aut_0(\Omega)$ is the almost direct product of $G$ and $N$. \end{lemma}

\subsection{$G$ contains a hyperbolic element}

\begin{lemma}\label{lem:limit_sets_coincide_1} With the notation above, 
\begin{align*}
\Lc(\Omega; G) = \Lc(\Omega; \Aut_0(\Omega))
\end{align*}
and $\Lc(\Omega; G)$ intersects at least two closed complex faces of $\partial\Omega$. In particular, $G$ contains an hyperbolic element.
\end{lemma}

\begin{proof} We first show that $\Lc(\Omega; G) = \Lc(\Omega; \Aut_0(\Omega))$. Suppose that $x \in \Lc(\Omega; \Aut_0(\Omega))$, then there exists $z \in \Omega$ and a sequence $g_m \in \Aut_0(\Omega)$ such that $g_m z \rightarrow x$. Now we can decompose $g_m = \overline{g}_m n_m$ where $\overline{g}_m \in G$ and $n_m \in N$. Since $N$ is compact, we can pass to a subsequence such that $n_m z \rightarrow w \in \Omega$. Then $\overline{g}_n w \rightarrow x$ by Proposition~\ref{prop:zero_dist_est}. Hence 
\begin{align*}
\Lc(\Omega; G) = \Lc(\Omega; \Aut_0(\Omega)).
\end{align*}

We now argue that $\Lc(\Omega; \Aut_0(\Omega))$, and hence $\Lc(\Omega; G)$, intersects at least two closed complex faces of $\partial\Omega$. Lemma~\ref{lem:non_compact} implies that $\Lc(\Omega; \Aut_0(\Omega))$ is non-empty. So suppose that $x \in \Lc(\Omega; \Aut_0(\Omega))$, then there exists $z \in \Omega$ and a sequence $g_n \in \Aut_0(\Omega)$ such that $g_n z \rightarrow x$. Now if one of the $g_n$ is hyperbolic, then we have nothing to show so assume that each $g_n$ is either elliptic or parabolic. By Theorem~\ref{thm:construct_hyp}, we can find a hyperbolic element $h \in \Aut(\Omega)$ such that 
\begin{align*}
H_h^+, H_h^-, T_{x}^{\Cb} \partial\Omega
\end{align*}
are all distinct. Then consider the elements $\varphi_{n,m} = h^m g_n h^{-m}$. Then Proposition~\ref{prop:non_hyp_attracting} and Proposition~\ref{prop:NS} imply that 
\begin{align*}
\lim_{n,m \rightarrow \infty} d_{\Euc}(\varphi_{n,m} z, H_h^+) = 0.
\end{align*}
But $\varphi_{n,m} \in \Aut_0(\Omega)$ and so we see that $\Lc(\Omega; \Aut_0(\Omega))$  intersects at least two closed complex faces of $\partial\Omega$. 
\end{proof}

\subsection{$G$ has real rank one and finite center}

In this subsection we will show that $G$ is a simple Lie group with real rank one and finite center.

Given $g \in G$, let $C(g)$ denote the centralizer of $g$ in $\Aut(\Omega)$. 

\begin{lemma}\label{lem:finite_centralizers_convex} With the notation above, if $h \in G$ is hyperbolic, then the quotient $C(h) / \{ h^n : n \in \Zb\}$ is compact. \end{lemma}

\begin{proof}
Consider a sequence $g_n \in C(h)$. We claim that we can find $n_k \rightarrow \infty$ and $m_k \in \Zb$ such that $g_{n_k}h^{m_k}$ converges in $G$. 

By Corollary~\ref{cor:alm_geod_shadow}, there exists an almost-geodesic $\sigma: \Rb \rightarrow \Omega$ such that 
\begin{align*}
R:=K_\Omega^{\Haus}\left( \{ h^m z_0 : m \in \Zb \}, \sigma(\Rb) \right)<\infty
\end{align*}
and
\begin{align*}
\lim_{t \rightarrow \pm \infty} d_{\Euc}\left(\sigma(t), H^{\pm}_{h} \right) = 0.
\end{align*}
Next consider the almost-geodesics $\sigma_n =g_n \sigma$. We claim that 
\begin{align*}
\lim_{t \rightarrow \pm \infty} d_{\Euc}\left(\sigma_n(t), H^{\pm}_{h} \right) = 0.
\end{align*}
By Proposition~\ref{prop:finite_dist} it is enough to show that
\begin{align*}
\limsup_{m \rightarrow \pm \infty} K_\Omega(h^m z_0 , g_n h^m z_0 )< \infty
\end{align*}
which follows from the fact that 
\begin{align*}
K_\Omega(h^m z_0, g_n h^m z_0 )= K_\Omega(h^m z_0 ,h^m g_n z_0 ) = K_\Omega(z_0, g_n z_0).
\end{align*}

Then by Theorem~\ref{thm:visible}, there exists $n_k \rightarrow \infty$ and $T_k \rightarrow \infty$ such that the almost-geodesics $t \rightarrow \sigma_{n_k}(t+ T_k)$ converges locally uniformly to an almost geodesic $\gamma: \Rb \rightarrow \Omega$. Further, there exists some $m_k \in \Zb$ such that 
\begin{align*}
R \geq K_\Omega( h^{m_k} z_0, \sigma(T_k)) = K_\Omega( g_{n_k}h^{m_k} z_0, \sigma_{n_k}(T_k)).
\end{align*}
So we can pass to a subsequence such that $g_{n_k} h^{m_k} z_0$ converges in $\Omega$. Since $\Aut(\Omega)$ acts properly on $\Omega$, we can pass to another subsequence such that $g_{n_k} h^{m_k}$ converges in $G$. Since $g_n$ was an arbitrary sequence in $C(h)$ we then see that $C(h) / \{ h^n : n \in \Zb\}$ is compact.
\end{proof}

\begin{lemma} With the notation above, $G$ has finite center. \end{lemma}

\begin{proof} Since $G$ is semisimple, the center of $G$ is discrete. So this follows immediately from Lemma~\ref{lem:finite_centralizers_convex}. \end{proof}

\begin{definition} An element $g \in G$ is \emph{$\Lc$-hyperbolic} (respectively \emph{$\Lc$-axial}, \emph{$\Lc$-elliptic}, \emph{$\Lc$-unipotent}) if $g$ is hyperbolic (respectively axial, elliptic, unipotent)  in $G$ in the Lie group sense (see Section~\ref{sec:basic_properties}).
\end{definition}

Fix a norm on $\gL$ and let $\norm{\cdot}$ be the associated operator norm on $\SL(\gL)$. 

\begin{lemma}\label{lem:dist_norm_est} With the notation above, if $z_0 \in \Omega$, then there exists some $\alpha \geq 1$ and $\beta \geq 0$ such that 
\begin{align*}
K_\Omega( g(z_0), z_0) \leq \alpha \log \norm{\Ad(g)} + \beta
\end{align*}
for all $g \in G$. 
\end{lemma}

\begin{proof}
This is identical to the proof of Lemma 5.11 in~\cite{Z2017b}.
\end{proof}

\begin{lemma}\label{lem:double_hyp} With the notation above, there exists an element $g \in G$ which is both hyperbolic and $\Lc$-hyperbolic. \end{lemma}

\begin{proof}
By Lemma~\ref{lem:limit_sets_coincide_1} there exists some $g \in G$ which is hyperbolic. Then by Proposition~\ref{prop:hyperbolic_QI}
\begin{align*}
\lim_{n \rightarrow \infty} \frac{K_\Omega(g^n(z), z) }{n} >0
\end{align*}
for all $z \in \Omega$. So by Lemma~\ref{lem:dist_norm_est} 
\begin{align}
\label{eq:ineq_norm_growth}
\liminf_{n \rightarrow \infty} \frac{\log \norm{\Ad(g)^n}}{n}  >0.
\end{align}
Using the Jordan decomposition, see Theorem~\ref{thm:jordan_decomp}, we can write $g = khu$ where $k$ is $\Lc$-elliptic, $h$ is $\Lc$-hyperbolic, $u$ is $\Lc$-unipotent, and $k,h, u$ commute. The inequality in~\eqref{eq:ineq_norm_growth} implies that $\Ad(h) \neq 1$. 

Now fix some $z_0 \in \Omega$. We claim that $ku$ is elliptic (as an element of $\Aut(\Omega)$). Since $\Ad(u)$ is unipotent and $\Ad(k)$ is elliptic we have
\begin{align*}
\lim_{n \rightarrow \infty} \frac{\log \norm{\Ad(ku)^n}}{n}  =0.
\end{align*}
So by Proposition~\ref{prop:hyperbolic_QI} and Lemma~\ref{lem:dist_norm_est}, we see that $ku$ is not hyperbolic. So if $ku$ were not elliptic, then $ku$ would be parabolic. But since $ku$ commutes with $g$, Proposition~\ref{prop:finite_dist} implies that
\begin{align*}
\lim_{n \rightarrow \pm \infty} d_{\Euc}((ku)^m g^n(z_0), H^\pm_g) =0
\end{align*}
for any $m \in \Nb$. So $ku$ cannot be parabolic by Proposition~\ref{prop:non_hyp_attracting}.

Now since $ku$ is elliptic, the set $\{ (ku)^nz_0 : n \in \Zb\}$ is relatively compact in $\Omega$. So 
\begin{align*}
\sup_{n \in \Zb} K_\Omega(h^n(z_0),g^n(z_0)) = \sup_{n \in \Zb}  K_\Omega(z_0,(ku)^n(z_0)) < \infty.
\end{align*}
So by Proposition~\ref{prop:finite_dist} 
\begin{align*}
\lim_{n \rightarrow \pm \infty} d_{\Euc}(h^n(z_0), H^\pm_g) =0.
\end{align*}
Thus $h$ is hyperbolic and $\Lc$-hyperbolic. 
\end{proof}

\begin{lemma} With the notation above, $G$ is a simple Lie group of non-compact type and has real rank one. \end{lemma}

\begin{proof} Pick an element $h \in G$ which is hyperbolic and $\Lc$-hyperbolic. By Proposition~\ref{prop:cartan_conj}, there exists a Cartan subgroup $A \leq G$ such that $h \in Z(G)A$. Then $Z(G)A \leq C(h)$ and so the quotient $Z(G)A / \{ h^n : n \in \Zb\}$ is compact. Since $A$ is isomorphic to $\Rb^r$ where $r = { \rm rank}_{\Rb}(G)$, this implies that $r=1$.\end{proof}

\subsection{The automorphism group has finitely many components}

In this section we show that $\Aut_0(\Omega)$ has finite index in $\Aut(\Omega)$. 

Since $G$ is a normal subgroup in $\Aut(\Omega)$, associated to every $g \in \Aut(\Omega)$ is an element $\tau(g) \in { \rm Aut}(G)$ defined by 
\begin{align*}
\tau(g)( h) = ghg^{-1}.
\end{align*}
Next let ${ \rm Inn}(G)$ denote the \emph{inner automorphisms of $G$}, that is the automorphisms of the form $h \rightarrow ghg^{-1}$ where $g \in G$. Then let ${\rm Out}(G) = \Aut(G)/{\rm Inn}(G)$. Finally define $[\tau]:\Aut(\Omega) \rightarrow {\rm Out}(G)$ by letting $[\tau](g)$ denote the equivalence class of $\tau(g)$. 

Since $G$ is semisimple, ${ \rm Out }(G)$ is finite (see for instance~\cite[Chapter X]{H2001}). So to prove that $\Aut_0(\Omega)$ has finite index in $\Aut(\Omega)$, it is enough to prove the following.

\begin{lemma}\label{lem:finite_comp}
With the notation above, $\Aut_0(\Omega)$ has finite index in $\ker [\tau]$. In particular, $\Aut_0(\Omega)$ has finite index in $\Aut(\Omega)$.
\end{lemma}

\begin{proof} It is enough to show that the quotient $\ker [\tau] / G$ is compact. So suppose that $g_n \in \ker [\tau]$ is a sequence. We claim that there exists $n_k \rightarrow \infty$ and $h_k \in G$ such that $g_{n_k} h_k$ converges in $\Aut(\Omega)$. Now for each $n \in \Nb$ there exists some $\overline{g}_n \in G$ such that $\tau(g_n) = \tau(\overline{g}_n)$. Then by replacing each $g_n$ with $g_n\overline{g}_n^{-1}$ we can assume that 
\begin{align*}
g_n g g_n^{-1} = g
\end{align*} 
for every $g \in G$ and $n \in \Nb$. Now fix a hyperbolic element $h \in G$. Then $g_n \in C(h)$ and so by Lemma~\ref{lem:finite_centralizers_convex} there exists $n_k \rightarrow \infty$ and $m_k \in \Zb$ such that $g_{n_k} h^{m_k}$ converges in $\Aut(\Omega)$. Since $g_n$ was an arbitrary sequence in $\ker [\tau]$ we see that $\ker [\tau] / G$ is compact. Hence $\Aut_0(\Omega)$ has finite index in $\ker [\tau]$.

\end{proof} 

\subsection{The limit set is a sphere}

In this subsection we show that $\Lc(\Omega)$ is homeomorphic to a sphere. We begin by observing that $\Lc(\Omega) = \Lc(\Omega; G)$.

\begin{lemma}\label{lem:limit_sets_coincide} With the notation above, $\Lc(\Omega) = \Lc(\Omega; G)$. 
\end{lemma}

\begin{proof} By Lemma~\ref{lem:limit_sets_coincide_1}, it is enough to show that $\Lc(\Omega) = \Lc(\Omega; \Aut_0(\Omega))$. Suppose that $x \in \Lc(\Omega)$. Then there exists $z \in \Omega$ and $\varphi_n \in \Aut(\Omega)$ such that $\varphi_n(z) \rightarrow x$. Since $\Aut_0(\Omega)$ has finite index in $\Aut(\Omega)$ we can pass to a subsequence and suppose that $\varphi_n = \phi_n g$ for some $\phi_n \in \Aut_0(\Omega)$ and $g \in \Aut(\Omega)$. Then $\varphi_n(gz) \rightarrow x$ and so $x \in \Lc(\Omega; \Aut_0(\Omega))$. 
\end{proof}

Fix an element $h \in G$ which is both hyperbolic and $\Lc$-hyperbolic. By Theorem~\ref{thm:hyp_ss},  there exists an one-parameter group $a_t$ of $\Lc$-hyperbolic elements such that $h = h_za_T$ for some $T >0$ and $h_z \in Z(G)$. 

\begin{lemma} With the notation above, $a:=a_1$ is hyperbolic.
\end{lemma}

\begin{proof} Since $H^\pm_{a_T} = H^{\pm}_a$ it is enough to show that $a_T$ is hyperbolic. Fix some $z_0 \in \Omega$. Since $Z(G)$ is finite, there exists some $M >0$ such that 
\begin{align*}
K_\Omega(h_z^n z_0, z_0) \leq M
\end{align*}
for all $n \in \Zb$. Then 
\begin{align*}
K_\Omega(a_T^n z_0, h^n z_0) = K_\Omega(z_0, h_z^{n} z_0) \leq M
\end{align*}
for all $n \in \Zb$. So by Proposition~\ref{prop:finite_dist} we see that $H^{\pm}_{a_T} = H^{\pm}_h$. So $a_T$ is hyperbolic. 
\end{proof}

Next let $K \leq G$ be a maximal compact subgroup associated to $a_t$ as in the discussion in Section~\ref{sec:polar_coor}. Then Theorem~\ref{thm:polar_coor} implies that every element $g \in G$ can be written as $g=k_1 a_t k_2$ for some $k_1, k_2 \in K$ and $t \in \Rb$.

The quotient $X:= G/K$ is simply connected and has a unique (up to scaling) negatively curved complete Riemannian metric, see Section~\ref{sec:basic_properties} for details. Let $d_X$ be the distance induced  by this Riemannian metric.

\begin{proposition}\label{prop:QI} With the notation above, if $z_0 \in \Omega$, then there exists $A \geq 1$ and $B \geq 0$ such that
\begin{align*}
\frac{1}{A}d_X( g_1K, g_2K) - B \leq K_\Omega( g_1z_0, g_2z_0) \leq A d_X( g_1K, g_2K) + B
\end{align*}
for all $g_1, g_2\in G$. 
\end{proposition}

\begin{proof}
By Theorem~\ref{thm:polar_coor}
\begin{equation}
\label{eq:QI1}
d_X( k a_t  K, K) = \abs{t}
\end{equation}
for all $t \in \Rb$ and $k \in K$. Further, by Proposition~\ref{prop:hyperbolic_QI} there exists some $\alpha, \beta$ such that 
\begin{equation}
\label{eq:QI2}
\frac{1}{\alpha} \abs{t}-\beta \leq K_\Omega(a_t(z_0), z_0) \leq \alpha \abs{t}+\beta.
\end{equation}
By compactness, there exists some $M \geq 0$ such that 
\begin{equation}
\label{eq:QI3}
K_\Omega( k z_0, z_0 ) \leq M
\end{equation}
for all $k \in K$. 

Now suppose that $g_1, g_2 \in G$. Then $g_2^{-1}g_1 = k_1 a_t k_2$ for some $k_1, k_2 \in K$ and $t \in \Rb$. Then
\begin{align*}
d_X( g_1 K, g_2K) =d_X( g_2^{-1}g_1 K, K)  = d_X(k_1 a_t K, K) = \abs{t}.
\end{align*}
Further
\begin{align*}
\abs{K_\Omega(g_1z_0, g_2z_0)-K_\Omega(a_t z_0, z_0)} \leq K_\Omega(k_2 z_0, z_0)+K_\Omega(z_0, k_1^{-1}z_0)
 \leq 2 M. 
\end{align*}
So Equations~\eqref{eq:QI1}, \eqref{eq:QI2}, and \eqref{eq:QI3} imply that
\begin{align*}
\frac{1}{A}d_X(g_1 K, g_2 K) -B \leq K_\Omega(g_1z_0, g_2z_0) \leq Ad_X(g_1 K, g_2 K) + B
\end{align*}
for some $A, B$ which do not depend on $g_1,g_2$.
\end{proof}

\begin{lemma}\label{lem:Fhyp_implies_Ahyp} With the notation above, if $g \in G$ is hyperbolic, then $g$ is $\Lc$-axial. \end{lemma}

\begin{proof} If $g \in G$ is hyperbolic, then Proposition~\ref{prop:hyperbolic_QI} implies that
\begin{align*}
\liminf_{n \rightarrow \infty} \frac{1}{n}K_\Omega(g^n(z_0), z_0) > 0.
\end{align*}
So Proposition~\ref{prop:QI} implies that
\begin{align*}
\liminf_{n \rightarrow \infty} \frac{1}{n}d_X(g^nK, K) > 0.
\end{align*}
Thus $g$ is $\Lc$-axial by Proposition~\ref{prop:translation_dist}.
\end{proof}

Let $X(\infty)$ be the geodesic boundary of $X$. Then, as described in Theorem~\ref{prop:ss_wolf_denjoy}, every $\Lc$-axial element $g \in G$ has an attracting  fixed point $\omega_g^+ \in X(\infty)$ and a repelling fixed point $\omega_g^- \in X(\infty)$. 

\begin{lemma} With the notation above, if $h_1, h_2 \in G$ are hyperbolic and $H_{h_1}^+ \neq H_{h_2}^+$, then $\omega_{h_1}^+ \neq \omega_{h_2}^+$. 
\end{lemma}

\begin{proof} If $\omega_{h_1}^+ = \omega_{h_2}^+$, then by Theorem~\ref{thm:rankone_hyp} there exists $m_k, n_k \rightarrow \infty$ such that 
\begin{align*}
\limsup_{k \rightarrow \infty} d_X( h_1^{m_k} K, h_2^{n_k} K) < \infty.
\end{align*}
Then by Proposition~\ref{prop:hyperbolic_QI} we have
\begin{align*}
\limsup_{k \rightarrow \infty} K_\Omega( h_1^{m_k} z_0, h_2^{n_k} z_0) < \infty.
\end{align*}
But since $H_{h_1}^+ \neq H_{h_2}^+$, this contradicts Proposition~\ref{prop:finite_dist}.
\end{proof}

\begin{lemma}\label{lem:unif_shadow} With the notation above,  let $r >0$ be such that 
\begin{align*}
x+ r \cdot { \bf n}_\Omega(x) \in \Omega
\end{align*}
for all $x \in \partial \Omega$. Then there exists $R > 0$ and a function $\psi: K \rightarrow \partial \Omega$ such that
\begin{align*}
K_\Omega^{\Haus}\left( \{ ka_t k^{-1} z_0 : t > 0 \}, \psi(k) + (0,r] \cdot { \bf n}_\Omega(\psi(k)) \right) \leq R
\end{align*}
for all $k \in K$. 
\end{lemma}

\begin{remark} By Theorem~\ref{thm:polar_coor}, every curve of the form $t \rightarrow ka_t K$ is a geodesic  in $X$ and by Proposition~\ref{prop:normal_lines} every curve of the form $t \rightarrow x + re^{-2t} { \bf n}_\Omega(x)$  with $x \in \partial \Omega$ is an almost-geodesic in $\Omega$. The above lemma shows that these two curves shadow each other. \end{remark}

\begin{proof} We first argue that for every $k \in K$ there exists $\psi(k) \in \partial \Omega$ such that 
\begin{align*}
\lim_{t \rightarrow \infty} ka_t k^{-1}(z) = \psi(k)
\end{align*}
for all $z \in \Omega$. If we knew that $ka_tk^{-1}$ was hyperbolic, this would follow from Theorem~\ref{thm:normal_line_shadowing}, Proposition~\ref{prop:cplx_disks}, and Proposition~\ref{prop:finite_dist}. Unfortunately, there doesn't seem to be an easy way to show that each $ka_tk^{-1}$ is hyperbolic (at this stage of the proof). 

By Theorem~\ref{thm:construct_hyp}, there exists some hyperbolic element $h \in G$ such that 
\begin{align*}
H^+_{h}, H^-_{h}, H^+_{ka_tk^{-1}}
\end{align*}
are all distinct. Now by Lemma~\ref{lem:Fhyp_implies_Ahyp}, $h$ is also $\Lc$-axial. Then by Theorem~\ref{thm:rankone_hyp} there exists an $\Ac$-hyperbolic element $g \in G$ such that 
\begin{align*}
d_X^{\Haus}( \{ g^n K : n \in \Nb\} , \{ ka_t k^{-1} K : t > 0 \}) < \infty
\end{align*}
and 
\begin{align*}
d_X^{\Haus}( \{ g^{-n} K : n \in \Nb\} , \{  h^n K : n \in \Nb\}) < \infty.
\end{align*}
So by Proposition~\ref{prop:QI} and Proposition~\ref{prop:finite_dist}, 
\begin{align*}
H^+_g = H^+_{ka_tk^{-1}} \text{ and } H^-_g = H^+_h.
\end{align*}
Thus $g$ is hyperbolic. Then by Theorem~\ref{thm:normal_line_shadowing}, Proposition~\ref{prop:cplx_disks}, and Proposition~\ref{prop:finite_dist} there exists some $x^+_g \in \partial \Omega$ such that 
\begin{align*}
\lim_{n \rightarrow \infty} g^n(z) = x^+_g
\end{align*}
for all $z \in \Omega$. Then by Proposition~\ref{prop:finite_dist} and Proposition~\ref{prop:cplx_disks} we have 
\begin{align*}
\lim_{t \rightarrow \infty} ka_t k^{-1}(z) = x^+_g
\end{align*}
for all $z\in \Omega$. So define $\psi(k):=x^+_g$. 

Then for $k \in K$ let $\sigma_k:[0,\infty) \rightarrow \Omega$ be the curve
\begin{align*}
\sigma_k(t) = \psi(k) + r e^{-2t} { \bf n}_\Omega(\psi(k)).
\end{align*}
Then by Propsoition~\ref{prop:normal_lines}, there exists some $\kappa > 0$ such that every $\sigma_k$ is an $(1,\kappa)$-almost-geodesic. \newline

\noindent \textbf{Claim:} There exists some $M_0 > 0$ such that 
\begin{align*}
K_\Omega( ka_t k^{-1} z_0, \sigma_k) \leq M_0
\end{align*}
for all $k \in K$ and $t > 0$. 

\begin{proof}[Proof of Claim:]
Suppose not, then for every $n \in \Nb$ there exists $k_n \in K$ and $t_n > 0$ such that 
\begin{align*}
K_\Omega\left( k_na_{t_n} k_n^{-1} z_0 , \sigma_{k_n} \right) > n.
\end{align*}
We can pass to a subsequence such that $\xi(k_n) \rightarrow x \in \partial \Omega$. By Theorem~\ref{thm:construct_hyp}, there exists some hyperbolic element $h \in G$ such that 
\begin{align*}
H^+_{h}, H^-_{h}, T_{x}^{\Cb} \partial \Omega
\end{align*}
are all distinct. Then by Theorem~\ref{thm:rankone_hyp} there exists some $R_1 > 0$ and a sequence $h_n$ of $\Lc$-hyperbolic elements such that 
\begin{align*}
d_X^{\Haus}( \{ h_n^m K: m \in \Nb\}, \{ k_na_{t} k_n^{-1} K : t > 0\}) \leq R_1
\end{align*}
and 
\begin{align*}
d_X^{\Haus}( \{ h_n^{-m} K: m \in \Nb\}, \{ h^{m} K : m \in \Nb \} )\leq R_1.
\end{align*}
Then by Proposition~\ref{prop:finite_dist}
\begin{align*}
H^+_{h_n} = H^+_{k_n a_{t_n} k_n^{-1}} = T^{\Cb}_{\psi(k_n)} \partial \Omega
\end{align*}
and 
\begin{align*}
H^-_{h_n} = H^+_{h}.
\end{align*}

Now there exists $m_n \in \Nb$ such that 
\begin{align*}
K_\Omega(h_n^{m_n}(z_0), k_na_{t_n} k_n^{-1} z_0) \leq R_1.
\end{align*}
Then let $\gamma_n = h^{-m_n}_n \sigma_{k_n}$. Then $\gamma_n$ is an $(1,\kappa)$-almost-geodesic and 
\begin{align*}
\lim_{n \rightarrow \infty} d_{\Euc}( \gamma_n(0), H^+_{h} ) = \lim_{n \rightarrow \infty} d_{\Euc}( h^{-m_n}_n(\sigma_{k_n}(0)), H^+_{h} )= 0.
\end{align*}
Further, 
\begin{align*}
K_\Omega^{\Haus} \left( \gamma_n, \{ h_n^{m-m_n}(z_0) : m > 0\} \right) < \infty
\end{align*}
and so by Proposition~\ref{prop:finite_dist} we have
\begin{align*}
0 = \lim_{t \rightarrow \infty} d_{\Euc}\left(\gamma_n(t), H_{h_n}^+\right) = \lim_{t \rightarrow \infty} d_{\Euc}\left(\gamma_n(t), H_{\psi(k_n)}^+\right).
\end{align*}

So by Theorem~\ref{thm:visible}, we can pass to a subsequence and find some $T_n \in [0,\infty)$ such that the almost-geodesics $t \rightarrow \gamma_n(t+T_n)$ converge locally uniformly to an almost-geodesic $\gamma_\infty$. But then 
\begin{align*}
\infty 
&= \lim_{n \rightarrow \infty} K_\Omega(k_na_{t_n} k_n^{-1} z_0, \sigma_{k_n}) \leq  \lim_{n \rightarrow \infty} K_\Omega(h_n^{m_n} z_0, \sigma_{k_n}(T_n))+R_1 \\
& = \lim_{n \rightarrow \infty} K_\Omega(z_0,h_n^{-m_n}  \sigma_{k_n}(T_n)) +R_1 = K_\Omega(z_0, \gamma_\infty(0)) +R_1< \infty
\end{align*}
so we have a contradiction. 
\end{proof}

Since $K$ is compact, there exists some $\delta >0$ such that 
\begin{align*}
K_\Omega( k a_t k^{-1} z_0, k a_{t+1} k^{-1} z_0)=K_\Omega(k^{-1} z_0, a_{1} k^{-1} z_0) \leq \delta
\end{align*}
for all $k \in K$ and $t \in \Rb$. 

Now fix some $k \in K$. For each $n \in \Nb$ there exists some $t_n \in [0,\infty)$ such that 
\begin{align*}
K_\Omega( ka_n k^{-1}(z_0), \sigma(t_n)) \leq M_0.
\end{align*}
Then since $\sigma_n$ is a $(1,\kappa)$-almost-geodesic we have 
\begin{align*}
\abs{t_n-t_{n+1}} \leq K_\Omega(\sigma(t_n), \sigma(t_{n+1})) + \kappa \leq 2M_0 + \delta + \kappa
\end{align*}
Since $t_n \rightarrow \infty$ we then see that for every $t \in [0,\infty)$ there exists some $t_n$ such that 
\begin{align*}
\abs{t-t_n} \leq M_0 + \delta/2+\kappa/2.
\end{align*}
So 
\begin{align*}
K_\Omega^{\Haus}\left( \{ ka_t k^{-1} z_0 : t > 0 \}, \psi(k) + (0,r] \cdot { \bf n}_\Omega(\psi(k)) \right) \leq 2M_0+\delta/2+3\kappa/2.
\end{align*}
Since $k \in K$ was arbitrary this completes the proof.
\end{proof}

Next let $M_a \leq K$ denote the elements in $K$ which commute with the subgroup $a_{\Rb}$. By the discussion in Section~\ref{sec:polar_coor}, $K/M_a$ is homeomorphic to $X(\infty)$. 

\begin{lemma} With the notation above, $\psi: K \rightarrow \partial \Omega$ is continuous and descends to a continuous map $\psi:K/M_a \rightarrow \partial \Omega$ which is a homeomorphism onto its image.
\end{lemma}

\begin{proof} Suppose that $k_n \rightarrow k$ in $K$. Since $\partial \Omega$ is compact, to show that $\psi(k_n) \rightarrow \psi(k)$ it is enough to show that every convergent subsequence of $\psi(k_n)$ converges to $\psi(k)$. So we can assume that $\psi(k_n) \rightarrow x \in \partial \Omega$. 

For each $m \in \Nb$ we have
\begin{align*}
\lim_{n \rightarrow \infty} k_n a_m k_n^{-1}(z_0) = k a_m k^{-1}(z_0).
\end{align*}
Further,
\begin{align*}
\lim_{m \rightarrow \infty} k a_m k^{-1}(z_0) = \psi(k).
\end{align*}
So we can find a sequence $m_n \rightarrow \infty$ such that 
\begin{align*}
\lim_{n \rightarrow \infty} k_n a_{m_n} k_n^{-1}(z_0) = \psi(k).
\end{align*}
On the other hand, there exists a sequence $\epsilon_n \rightarrow 0$
\begin{align*}
\lim_{n \rightarrow \infty} K_\Omega(k_n a_{m_n} k_n^{-1}(z_0), \psi(k_n) + \epsilon_n \cdot {\bf n} (\psi(k_n)) ) \leq R.
\end{align*}
Since $\psi(k_n) + \epsilon_n \cdot {\bf n} (\psi(k_n)) \rightarrow x$, Propositions~\ref{prop:finite_dist} and~\ref{prop:cplx_disks} imply that $x = \psi(k)$. So $\psi(k_n) \rightarrow \psi(k)$ and thus $\psi: K \rightarrow \partial \Omega$ is continuous.

Next suppose that $k_1 M_a = k_2 M_a$ then 
\begin{align*}
k_2 a = k_1 a k_2^\prime
\end{align*}
for some $k_2^\prime \in M_a$. So
\begin{align*}
K_\Omega( k_1 a_n k_1^{-1}(z_0),  & k_2 a_n k_2^{-1}(z_0)) = K_\Omega( k_1 a_n k_1^{-1}(z_0),  k_1 a^n k_2^\prime k_2^{-1}(z_0)) \\
& = K_\Omega( z_0, k_2^\prime k_2^{-1}(z_0))
\end{align*}
so by Proposition~\ref{prop:finite_dist} and Proposition~\ref{prop:cplx_disks}, we see that $\psi(k_1) = \psi(k_2)$. Thus $\psi$ descends to a continuous map $\psi: K / M_a \rightarrow \partial \Omega$. 

Now suppose that $\psi(k_1) = \psi(k_2)$. Then by Proposition~\ref{prop:QI} and Lemma~\ref{lem:unif_shadow} we have
\begin{align*}
d_X^{\Haus}\left( \{ k_1 a_t K : t > 0\}, \{ k_2 a_t K : t >0\} \right) < \infty.
\end{align*}
But then $k_1M_a = k_2M_a$ by Theorem~\ref{thm:polar_coor}. So the map $\psi : K/M_a \rightarrow \partial \Omega$ is injective. Since $K/M_a$ is compact, the map $\psi : K/M_a \rightarrow \partial \Omega$ is a homeomorphism onto its image.

\end{proof}

\begin{lemma}\label{lem:cont_ext} With the notation above,  $\psi(K) = \Lc(\Omega)$. In particular, $\Lc(\Omega)$ is homeomorphic to $X(\infty)$, a sphere of dimension $\dim X - 1$. \end{lemma}

\begin{proof} Since $G$ has real rank one, $K/M_a$ is homeomorphic to a $X(\infty)$, see Section~\ref{sec:polar_coor}. So the ``in particular'' part will follow from the first claim and the previous lemma. 

Suppose that $x \in \Lc(\Omega)$. Then $x \in \Lc(\Omega; G)$ by Lemma~\ref{lem:limit_sets_coincide}. So there exists some $g_n \in G$ and some $w \in \Omega$ such that $g_n w \rightarrow x$. 

Now $g_n = k_n a_{t_n} \ell_n$ for some $k_n, \ell_n \in K$ and $t_n > 0$, see Theorem~\ref{thm:polar_coor}. By passing to a subsequence we can suppose that $k_n \rightarrow k_\infty$. If 
\begin{align*}
M:=\sup_{k \in K} K_\Omega(kw,w),
\end{align*}
then
\begin{align*}
K_\Omega(k_n a_{t_n} k_n^{-1} w, g_n w) = K_\Omega( w, k_n \ell_n w) \leq M.
\end{align*}
So by Lemma~\ref{lem:unif_shadow} there exists some $s_n >0$ such that 
\begin{align*}
K_\Omega\left(g_nw, \psi(k_n) + r e^{-2s_n} {\bf n}_\Omega(\psi(k_n))\right) \leq R+M+K_\Omega(z,z_0)
\end{align*}
for all $n \in \Nb$. Then since $s_n \rightarrow \infty$ we have
\begin{align*}
\lim_{n \rightarrow \infty} \psi(k_n) + r e^{-2s_n} {\bf n}_\Omega(\psi(k_n)) = \psi(k_\infty).
\end{align*}
Then by Proposition~\ref{prop:finite_dist} and Proposition~\ref{prop:cplx_disks} we have
\begin{align*}
x=\lim_{n \rightarrow \infty} g_nw = \psi(k_\infty).
\end{align*}
So $x \in \psi(K)$. Since $x \in \Lc(\Omega)$ was arbitrary, we see that $\psi(K) = \Lc(\Omega)$.

 \end{proof}

\subsection{Dimension of the limit set}\label{sec:dim}

In this section we prove that either 
 \begin{enumerate}
 \item $\dim \Lc(\Omega) \leq \dim \partial \Omega - 2$ or
 \item $\Lc(\Omega) = \partial \Omega$ and $\Omega$ is biholomorphic to the unit ball. 
 \end{enumerate}

Fix some $z_0 \in \Omega$. Since $G$ acts properly on $\Omega$, $G \cdot z_0$ is a closed topological submanifold of $\Omega$. We first observe that
\begin{align*}
\dim G \cdot z_0 \geq \dim \Lc(\Omega)+1.
\end{align*}
To see this let $K_{z_0}$ denote the stabilizer of $z_0$. Since $G$ acts properly on $\Omega$, $K_{z_0}$ is a compact subgroup. Now let $K$ be a maximal compact subgroup of $G$. Then $K_{z_0}$ is conjugate to a subgroup of $K$, see for instance~\cite[Chapter VI, Theorem 2.1]{H2001}. Thus $\dim K_{z_0} \leq \dim K$. Then 
\begin{align}
\label{eq:ineq}
\dim \Lc(\Omega) +1 & =\dim X(\infty)+1 =  \dim X = \dim G-\dim K \\
& \leq \dim G-\dim K_{z_0} = \dim G \cdot z_0. \nonumber
\end{align}

\noindent \textbf{Case 1:} $\dim \Lc(\Omega) = 2d-1$. Then $\dim G \cdot z_0=2d$. Then $G \cdot z_0$ is an open, closed, and connected subset of $\Omega$. Hence $G \cdot z_0 = \Omega$. We next claim that $\Omega$ is a bounded symmetric domain. Since $G$ is a simple Lie group acting transitively on $\Omega$ there are many ways to establish this. Here is one argument: since $G$ is a simple Lie group, a theorem of Borel implies that $G$ has a cocompact lattice $\Gamma \leq G$. Then since $G$ acts transitively on $\Omega$, the group $\Gamma$ acts cocompactly on $\Omega$. Then by a theorem of Frankel~\cite{F1989}, $\Omega$ is a bounded symmetric domain. Finally, since $G$ has real rank one, the classification of all bounded symmetric domains implies that $\Omega$ is biholomorphic to the ball. \newline

\noindent \textbf{Case 2:} $\dim \Lc(\Omega) = 2d-2$. Then $\dim G \cdot z_0 \geq  2d-1$. If $\dim G \cdot z_0 = 2d$ then the argument in Case (1) implies that $\Omega$ is biholomorphic to the ball. So assume that $\dim G \cdot z_0 = 2d-1$. Then Equation~\eqref{eq:ineq} implies that 
\begin{align*}
\dim K_{z_0} = \dim K.
\end{align*}
Further, $\dim X= \dim \Lc(\Omega)+1=2d-1$ is odd. So, by the classification of negatively curved symmetric spaces, we see that $X$ is isomorphic to real hyperbolic $(2d-1)$-space. Thus $K$ is locally isomorphic to $\SO(2d-1)$. So 
\begin{align*}
\dim K_{z_0} = \dim K = \dim \SO(2d-1)= \frac{2d(2d-1)}{2} = d(2d-1).
\end{align*}

Next consider the homomorphism $\rho:K_{z_0} \rightarrow \GL_d(\Cb)$ given by $\rho(k) = d_{\Cb}(k)_{z_0}$. Since $\Aut(\Omega)$ preserves the Bergman metric, a complete Riemannian metric on $\Omega$, any $\varphi \in \Aut(\Omega)$ is determined by $\varphi(z_0)$ and $d_{\Cb}(\varphi)_{z_0}$, see~\cite[Chapter 1, Lemma 11.2]{H2001}. So $K_{z_0}$ is isomorphic to $\rho(K_{z_0})$. However, since $\rho(K_{z_0})$ preserves the Bergman metric at $z_0$ we then see that $\rho(K_{z_0})$ is conjugate to a subgroup of ${ \rm U}(d)$. So 
\begin{align*}
\dim K_{z_0} = \dim \rho(K_{z_0}) \leq \dim { \rm U}(d) = d^2.
\end{align*}
So
\begin{align*}
2d-1 \leq d
\end{align*}
which is only possible if $d=1$. But $\Omega$, being convex, is simply connected. So by the Riemman mapping theorem $\Omega$ is biholomorphic to the disk and so $G \cdot z_0 = \Omega$ which contradicts our assumption that $\dim \Lc(\Omega)=2d-2$. 

\appendix 

 \section{Semisimple Lie groups and symmetric spaces}\label{sec:basic_properties}

In the proofs of Theorems~\ref{thm:main_convex}, we use some basic properties about semisimple Lie groups and the symmetric spaces they act on. In this section we recall these properties and give references. 

For the rest of the section we make the following assumption.

\begin{assumption} $G$ is a connected semisimple Lie group with finite center. \end{assumption}

Let $\gL$ be the Lie algebra of $G$. Then there is a Lie algebra decomposition 
\begin{align*}
\gL=\gL_1 \oplus \dots \oplus \gL_n
\end{align*}
into simple Lie subalgebras, see for instance~\cite[Chapter 1, Theorem 1.54]{K2002}). Then let $G_i$ be the connected subgroup of $G$ generated by $\exp(\gL_i)$. The following fact is standard (for a proof see for instance~\cite[Lemma A.1]{Z2017b}). 

\begin{lemma} Each $G_i$ is a closed subgroup of $G$ and $G$ is the almost direct product of $G_1,\dots, G_n$. \end{lemma} 

We now make the additional assumption that $G$ has no compact factors, more precisely: 

\begin{addassumption} Every $G_i$ is non-compact. \end{addassumption}

Next let $\Ad: G \rightarrow SL(\gL)$ denote the adjoint representation. The kernel of $\Ad$ is just the center of $G$, denoted $Z(G)$, and so we have an isomorphism $G/Z(G) \cong \Ad(G)$. 

\begin{definition}
We then say an element $g \in G$ is:
\begin{enumerate}
\item \emph{semisimple} if $\Ad(g)$ is diagonalizable in $\SL(\gL^{\Cb})$,
\item \emph{hyperbolic} if $\Ad(g)$ is diagonalizable in $\SL(\gL)$ with all positive eigenvalues,
\item \emph{unipotent} if $\Ad(g)$ is unipotent in $\SL(\gL)$, and
\item \emph{elliptic} if $\Ad(g)$ is elliptic in $\SL(\gL)$.
\end{enumerate}
\end{definition}

Since $G$ is semisimple, every element can be decomposed into a product of a elliptic, hyperbolic, and unipotent element. More precisely:

\begin{theorem}[Jordan Decomposition]\label{thm:jordan_decomp} If $g \in G$, then there exists $g_e, g_h, g_u \in G$ such that 
\begin{enumerate}
\item $g=g_eg_hg_u$,
\item $g_e \in G$ is elliptic, $g_h \in G$ is hyperbolic, $g_u \in G$ is unipotent, and 
\item $g_e,g_h,g_u$ commute. 
\end{enumerate}
Moreover, the $g_e,g_h,g_u$ are unique up to factors in $\ker \Ad = Z(G)$.
\end{theorem}

\begin{proof} See for instance~\cite[Theorem 2.19.24]{E1996}. \end{proof}

A subgroup $A \leq G$ is called a \emph{Cartan subgroup} if $A$ is closed, connected, abelian, and every element in $A$ is hyperbolic. The \emph{real rank of $G$}, denoted by ${\rm rank}_{\Rb}(G)$, is defined to be
\begin{align*}
{\rm rank}_{\Rb}(G) = \max \{ \dim A: A \text{ is a Cartan subgroup of } \Ad(G)\}.
\end{align*}
We will need the following fact about Cartan subgroups.

\begin{proposition}\label{prop:cartan_conj} If $g \in G$ is hyperbolic and $A \leq G$ is a maximal Cartan subgroup, then $g$ is conjugate to an element of $Z(G)A$. 
\end{proposition}

\begin{proof} See for instance~\cite[Chapter IX, Theorem 7.2]{H2001}. \end{proof}

Now fix $K \leq G$ a maximal compact subgroup. Then the quotient manifold $X=G/K$ is diffeomorphic to $\Rb^{\dim X}$ and has a unique (up to scaling) non-positively curved $G$-invariant Riemannian metric $g$, see~\cite[Section 2.2]{E1996} for details. Let $d_X$ denote the distance induced by $g$.

Using the Jordan decomposition we make the following definition. 

\begin{definition} If $g \in G$ has a Jordan decomposition $g=e_gh_gu_g$ with $\Ad(h_g) \neq 1$ and $u_g=1$ then we say that $g$ is \emph{axial}.
\end{definition}

Notice that every hyperbolic element is obviously axial. We can describe the action of axial and hyperbolic elements on $X$ as follows.

\begin{theorem}\label{thm:axial} With the notation above, $g \in G$ is axial if and only if there exists a geodesic $\sigma: \Rb \rightarrow X$ such that $g\sigma(t) = \sigma(t+T)$ for some $T > 0$. 
\end{theorem}

\begin{proof} See for instance~\cite[Proposition 2.19.18]{E1996}. \end{proof}

\begin{theorem}\label{thm:hyp_ss} With the notation above, if $g \in G$ is hyperbolic, then:
\begin{enumerate}
\item there exists a one-parameter subgroup $g_t$ of hyperbolic elements such that $g \in Z(G)\{ g_{t} : t \in \Rb\}$, and
\item there exists some point $x_0 \in X$ such that the curve $t \rightarrow g_t(x_0)$ is a geodesic in $(X,d_X)$. 
\end{enumerate}
Conversely, for any geodesic $\sigma : \Rb \rightarrow X$ there exists a one-parameter subgroup $h_t$ of hyperbolic elements such that $h_t(\sigma(s)) = \sigma(s+t)$ for all $s,t \in \Rb$.
\end{theorem}

\begin{proof} See for instance~\cite[Proposition 2.19.18]{E1996}. \end{proof}

We now focus on the real rank one case. 

\begin{addassumption} ${\rm rank}_{\Rb}(G)=1$. \end{addassumption}

Since 
\begin{align*}
{ \rm rank}_{\Rb}(G) = \sum_{i=1}^n { \rm rank}_{\Rb}(G_i)
\end{align*}
this implies that $G$ is a simple Lie group. In addition, by the classification of simple Lie groups, $G$ is locally isomorphic to one of $\SO(k,1)$, $\SU(k,1)$, $\Sp(k,1)$, or $F^{-20}_{4}$. Further, the associated symmetric space $(X,d_X)$ is either a real hyperbolic space, a complex hyperbolic space, a quaternionic hyperbolic space, or the Cayley-hyperbolic plane. In all these cases, $(X,d_X)$ is a negatively curved Riemannian manifold. For details see~\cite[Chapter 19]{M1973}.

Since $X$ is a non-positively curved simply connected Riemannian manifold, there exists a compactification called the \emph{geodesic compactification} which can be defined as follows. Let $\Gc$ denote the set of unit speed geodesic rays $\sigma:[0,\infty) \rightarrow X$. Then we say two geodesics $\sigma_1, \sigma_2 \in \Gc$ are equivalent if 
\begin{align*}
\lim_{t \rightarrow \infty} d_X(\sigma_1(t), \sigma_2(t)) < \infty.
\end{align*}
Finally let $X(\infty) = \Gc / \sim$. This gives a compactification $\overline{X} = X \cup X(\infty)$ of $X$ as follows. First fix a point $x_0 \in X$. Since $X$ is non-positively curved, for any $x \in X$ there exists a unique geodesic segment $\sigma_x$ joining $x_0$ to $x$. We then say that a sequence $x_n \in X$ converges to a point $\sigma \in X(\infty)$ if the geodesic segments $\sigma_{x_n}$ converge locally uniformly to $\sigma$. This construction does not depend on the initial choice of $x_0$. See~\cite[Section 1.7]{E1996} for details. 

Since $G$ acts by isometries on $X$ and the construction of $X(\infty)$ is independent of base point, the action of $G$ on $X$ extends to an action on $X \cup X(\infty)$. In general this action is only continuous, but for negatively curved symmetric spaces we have the following. 

\begin{theorem} With the notation above, $\overline{X}$ has a smooth structure, with this structure $X(\infty)$ is diffeomorphic to a sphere of dimension $\dim X -1$, and the action of $G$ on $X$ extends to a smooth action on $X(\infty)$. \end{theorem}

This theorem follows from considering the standard models of the negatively curved symmetric spaces, see~\cite[Chapter 19]{M1973}.

Given two geodesic rays $\sigma_1, \sigma_2: [0,\infty) \rightarrow X$ the function 
\begin{align*}
f(t) = d_X(\sigma_1(t), \sigma_2(t))
\end{align*}
is convex, see~\cite[Chapter II, Proposition 2.2]{BH1999}. So, if $x_n \in X$ is a sequence converging to some $\xi \in X(\infty)$ and $y_n \in X$ is a sequence with 
\begin{align*}
\sup_{n \in \Nb} d_X(x_n, y_n) < \infty,
\end{align*}
 then $y_n$ converges to $\xi$ as well. This fact combined with Theorem~\ref{thm:axial} implies the following.

\begin{proposition}\label{prop:ss_wolf_denjoy} With the notation above, if $g \in G$ is axial, then there exists distinct points $\omega^{+}_g, \omega^-_g \in X(\infty)$ such that 
\begin{align*}
\lim_{n \rightarrow \infty} g^n(x) = \omega_g^+ \text{ and } \lim_{n \rightarrow -\infty} g^n(x) = \omega_g^-
\end{align*}
for all $x \in X$. 
\end{proposition}

Since $G$ has real rank one, there is simple characterization of axial elements. Given an element $g \in G$ we define the \emph{translation length} of $g$ to be
\begin{align*}
\tau(g) = \lim_{n \rightarrow \infty} \frac{d_X(g^n(x),x)}{n}.
\end{align*}
Since 
\begin{align*}
d_X(g^{m+n}(x),x) \leq d_X(g^m(x),x) + d_X(g^n(x),x)
\end{align*}
this limit exists by a standard lemma (see for instance~\cite[Theorem 4.9]{W1982}). Further the limit does not depend on $x$. 

One can then show the following. 

\begin{proposition}\label{prop:translation_dist}
With the notation above, suppose $g \in G$. Then $\tau(g) >0$ if and only if $g$ is axial. 
\end{proposition}

\begin{proof}[Proof Sketch]
Fix a norm on $\gL$ and let $\norm{\cdot}$ be the associated operator norm on $\GL(\gL)$. Then there exists $A \geq 1$ and $B \geq 0$ such that 
\begin{align*}
\frac{1}{A} \log \norm{\Ad(g)} - B \leq d_X( gx,x) \leq A \log \norm{\Ad(g)} + B
\end{align*}
for all $g \in G$ by Theorem~\ref{thm:polar_coor} Part (2). So $\tau(g) >0$ if and only if 
\begin{align*}
\liminf_{n \rightarrow \infty} \frac{1}{n}  \log \norm{\Ad(g^n)} > 0.
\end{align*} 

Further if $g=g_e g_h g_u$ is a Jordan decomposition of $g$, then 
\begin{align*}
\liminf_{n \rightarrow \infty} \frac{1}{n}  \log \norm{\Ad(g^n)} > 0 
\end{align*} 
if and only if $\Ad(g_h) \neq 1$. So $\tau(g) > 0$ if and only if $\Ad(g_h) \neq 1$. 

Since $X$ is negatively curved, if $h \in G$ is a hyperbolic element then the quotient $C_G(h) / \{ h^n : n \in \Zb\}$ is compact (see the proof of~\cite[Chapter III, Corollary 3.10]{BH1999}). So if $\Ad(g_h) \neq 1$, then $\Ad(u_h) =1$. So $\tau(g) >0$ if and only if $g$ is axial. 
\end{proof}

Since $G$ has rank one, we also have the following.

\begin{theorem}\label{thm:rankone_hyp} With the notation above:
\begin{enumerate}
\item If $\xi,\eta \in X(\infty)$ are distinct, then there exists an unique (up to reparametrization) geodesic $\gamma: \Rb \rightarrow X$ such that 
\begin{align*}
\lim_{t \rightarrow \infty} \gamma(t) = \xi \text{ and } \lim_{t \rightarrow -\infty} \gamma(t) = \eta.
\end{align*}
\item If $\xi,\eta \in X(\infty)$ are distinct, then there exists a hyperbolic element $h$ such that $\tau(h)=1$, $\omega^+_{h} = \xi,$ and $\omega^-_{h} = \eta$. 
\item Let $x_0 \in X$ and $h_1, h_2 \in G$ be hyperbolic elements with $\omega^+_{h_1} = \omega^+_{h_2}$. Then
\begin{align*}
d_X^{\Haus}(\{ h_1^n(x_0) : n \in \Nb\}, \{ h_2^n(x_0): n \in \Nb\} ) 
\end{align*}
is finite and depends continuously on $x_0 \in X$ and $h_1, h_2 \in G$.
\end{enumerate}
\end{theorem}

\begin{proof} For part (1) see for instance~\cite[Proposition 4.4]{EO1973}. Then part (2) and (3) follows from part (1) and Theorem~\ref{thm:hyp_ss}. \end{proof}

\subsection{Polar coordinates}\label{sec:polar_coor} When $G$ has real rank one, every hyperbolic element induces ``polar coordinates'' on $X$. In particular, let $a \in G$ be hyperbolic. After possibly replacing $a$ with an element in $aZ(G)$,  Theorem~\ref{thm:hyp_ss} implies that there exists a one-parameter subgroup $a_t$ of hyperbolic elements such that $a=a_T$ for some $T > 0$. Further, there exists a geodesic $\gamma : \Rb \rightarrow X$ such that $a_t(\gamma(s)) = \gamma(s+t)$ for all $s,t \in \Rb$. 

Let $K_{0}$ denote the stabilizer of $x_0=\gamma(0)$ in $G$. Then since $G$ acts transitively on $X$, we have a natural identification $G/K_{0} = X$. Next let $M_a \leq K_0$ denote the elements of $K_0$ that commute with the subgroup $a_{\Rb}$. We then have the following. 

\begin{theorem}\label{thm:polar_coor} With the notation above: 
\begin{enumerate}
\item  $G = K_{0} a_{\Rb} K_{0}$,
\item For any $k \in K_{0}$ the curve $\gamma_k(t) = k a_t x_0$ is a geodesic in $X$ and so
\begin{align*}
d_X( ka_t K_0, K_0) = \abs{t}.
\end{align*}
\item If $k_1, k_2 \in K_0$ and 
\begin{align*}
\lim_{t \rightarrow \infty} d_X(\gamma_{k_1}(t), \gamma_{k_2}(t)) < \infty,
\end{align*}
then $k_1 M_a= k_2 M_a$ (and hence $\gamma_{k_1}=\gamma_{k_2}$).
\item The map $kM_a \in K_0/M_a \rightarrow [\gamma_k] \in X(\infty)$ is a homeomorphism. 
\end{enumerate}
\end{theorem}
 
\begin{proof} Part (1) is the so-called \emph{Cartan decomposition}, see for instance~\cite[Chapter IX, Theorem 1.1]{H2001}. By our choice of $x_0$, the curve $\gamma(t) = a_t x_0$ is a geodesic in $X$. Since $G$ acts by isometries, each $\gamma_k$ is also a geodesic and part (2) is true.

 Now suppose that $k_1, k_2 \in K$. Then since $X$ is non-positively curved the function 
 \begin{align*}
 f(t) = d_X(\gamma_{k_1}(t), \gamma_{k_2}(t))
 \end{align*}
 is convex (see~\cite[Chapter II, Proposition 2.2]{BH1999}) and $f(0)=0$. So if 
 \begin{align*}
\lim_{t \rightarrow \infty} d_X(\gamma_{k_1}(t), \gamma_{k_2}(t)) < \infty,
\end{align*}
then $\gamma_{k_1}(t) = \gamma_{k_2}(t)$ for all $t$. Then ~\cite[Chapter IX, Corollary 1.2]{H2001} implies that $k_1 M_a = k_2 M_a$. Thus part (3) is true. 

Finally, part (4) follows from part (3) and the definition of $X(\infty)$. 
\end{proof}

\bibliographystyle{alpha}
\bibliography{complex_kob}

\end{document}